\documentclass[final,reqno]{elsarticle}
\usepackage{mathrsfs}
\usepackage{rotating} 

\usepackage{amssymb,amsmath,bbm,amsfonts,latexsym,subfigure}
\usepackage{amsthm}
\usepackage{graphicx,float,epsfig,color,fancyhdr}
\newtheorem{lemma}{Lemma}[section]
\newtheorem{remark}{Remark}[section]
\newtheorem{proposition}{Proposition}[section]
\newtheorem{theorem}{Theorem}[section]
\newtheorem{corollary}{Corollary}[section]
\newtheorem{problem}{Problem}
\renewcommand{\theequation}{\thesection.\arabic{equation}}

\def\bA{\mathbf{A}}

\def\bu{\mathbf{u}}
\def\bv{\mathbf{v}}
\def\bw{\mathbf{w}}
\def\bff{\mathbf{f}}

\def\beps{\boldsymbol{\varepsilon}}

\def\bO{\boldsymbol{0}}
\def\bp{\mathbf{p}}
\def\bta{\boldsymbol{\tau}}

\def\bpi{\boldsymbol{\Pi}}
\def\bI{\boldsymbol{I}}

\def\bsi{\boldsymbol{\sigma}}
\def\bta{\boldsymbol{\tau}}

\def\be{\boldsymbol{e}}
\def\bP{\mathbf{P}}
\def\bV{\boldsymbol{\mathcal{V}}}

\def\bW{\boldsymbol{\mathcal{W}}}

\def\CE{\mathcal{E}}
\def\CS{\mathcal{S}}

\def\CT{\mathcal{T}}

\def\Cten{\boldsymbol{\mathcal{C}}}

\def\dim{\mathop{\mathrm{\,dim}}\nolimits}
\def\disp{\displaystyle}
\def\Div{\mathop{\mathbf{div}}\nolimits}

\def\G{\Gamma}

\def\hCT{\widehat{\CT}}

\def\hdel{\widehat{\delta}}

\def\H{\mathrm{H}}

\def\HuO{{\H^1(\O)}}

\def\HurO{{\H^{1+r}(\O)}}

\def\l{\lambda}

\def\LO{{\mathrm{L}^2(\O)}}

\def\N{{\mathbb{N}}}

\def\O{\Omega}

\def\PiE{\Pi_{\varepsilon}^E}
\def\bpisi{\boldsymbol{\PiE}}

\def\R{{\mathbb{R}}}

\def\tE{\widetilde{\CE}}

\def\tr{\mathop{\mathrm{tr}}\nolimits}

\def\T{{\mathcal T}}

\def\vv{v}

\def\PiO{\Pi_{0}^{E}}
\def\bpiO{\boldsymbol{\PiO}}

\renewcommand\H{\mathrm{H}}

\newcommand\0{\boldsymbol{0}}
\newcommand\bn{\boldsymbol{n}}
\newcommand\bT{\mathbf{T}}
\renewcommand\sp{\mathop{\mathrm{sp}}\nolimits}

\newcommand\bbP{\mathbb{P}}
\def\tr{\mathop{\mathrm{tr}}\nolimits}

\allowdisplaybreaks
\journal{}
\date{\today}
\begin{document}
\begin{frontmatter}

\title{A priori and a posteriori error estimates for
a virtual element spectral analysis for the elasticity equations}

\author[1,2]{David Mora}
\ead{dmora@ubiobio.cl}
\address[1]{Departamento de Matem\'atica,
Universidad del B\'io-B\'io, Casilla 5-C, Concepci\'on, Chile.}
\address[2]{Centro de Investigaci\'on en Ingenier\'ia Matem\'atica
(CI$^2$MA), Universidad de Concepci\'on, Concepci\'on, Chile.}
\author[3]{Gonzalo Rivera}
\ead{gonzalo.rivera@ulagos.cl}
\address[3]{Departamento de Ciencias Exactas,
Universidad de Los Lagos, Casilla 933, Osorno, Chile.}

\begin{abstract} 
We present a priori and a posteriori error analysis
of a Virtual Element Method (VEM) to
approximate the vibration frequencies and modes of an elastic solid.
We analyze a variational formulation relying only on the solid displacement
and propose an $H^{1}(\O)$-conforming discretization by means of VEM.
Under standard assumptions on the computational domain, we show that
the resulting scheme provides a correct approximation of the spectrum
and prove an optimal order error estimate for the eigenfunctions and a
double order for the eigenvalues. Since, the VEM
has the advantage of using general polygonal meshes,
which allows implementing efficiently mesh refinement strategies,
we also introduce a residual-type
a posteriori error estimator and prove its reliability and efficiency.
We use the corresponding error estimator to drive an adaptive scheme.
Finally, we report the results of a  couple of numerical tests
that allow us to assess the performance of this approach.


\end{abstract}


\date{\today}
\end{frontmatter}


\setcounter{equation}{0}
\section{Introduction}
\label{SEC:INTR}


We analyze in this paper a {\it Virtual Element Method}
for an eigenvalue problem arising in linear elasticity.
The Virtual Element Method (VEM), recently introduced in
\cite{BBCMMR2013,BBMR2014}, is a generalization of the
Finite Element Method, which is characterized by the capability
of dealing with very general polygonal/polyhedral meshes. 
In recent years, the interest in numerical methods that can make
use of general polygonal/polyhedral meshes for the numerical
solution of partial differential equations has undergone a significant growth;
this because of the high flexibility that this kind of meshes allow
in the treatment of complex geometries. Among the large number of
papers on this subject, we cite as a minimal
sample~\cite{BLMbook2014,CGH14,DPECMAME2015,DPECRAS2015,ST04,TPPM10}.

Although  VEM  is very recent, it has been applied to a large
number of  problems; for instance, to Stokes, Brinkman,
Cahn-Hilliard, plates bending, advection-diffusion,
Helmholtz, parabolic, and hyperbolic
problems have been introduced in
\cite{ABMV2014,ABSV2016,BLV-M2AN,BMR2016,BM12,BBBPS2016,CG16,CGS17,ChM-camwa,PPR15,vacca1,V-m3as18,vacca2}. Regarding 
VEM for linear and non-linear
elasticity we mention \cite{BBM,BLM2015,Paulino-VEM,WRR2016},  for spectral problems \cite{BMRR,GVXX,MRR2015,MRV},
whereas a posteriori error analysis for VEM have been developed
in \cite{BMm2as,BeBo2017,CGPS,MRR2}.





The numerical approximation of eigenvalue problems for partial
differential equations is object
of great interest from both, the practical and theoretical points of
view, since they appear in many applications. We refer to \cite{Boffi,BGG2012}
and the references therein for the state of the art in this subject area.
In particular, this paper focus on the approximation by VEM of the
vibration frequencies and modes of an elastic solid. One motivation for
considering this problem is that it constitutes a stepping stone towards
the more challenging goal of devising virtual element spectral
approximations for coupled systems involving fluid-structure
interaction, which arises in many engineering problems
(see \cite{BGHRS2008} for a thorough discussion on this topic).
Among the existing techniques to solve this problem,
various finite element methods have been proposed
and analyzed in different frameworks for
instance in the following references~\cite{BO,BHPR2001,Hernadez2009,MMR2013}.


On the other hand, in numerical computations it is important
to use adaptive mesh refinement strategies based on a posteriori
error indicators. For instance, they guarantee achieving errors
below a tolerance with a reasonable computer cost in presence of
singular solutions. Several approaches have been
considered to construct error estimators based on the residual
equations (see \cite{ATO,Verfurth} and the references therein).
Due to the large flexibility of the meshes to which the
VEM is applied, mesh adaptivity becomes an
appealing feature since mesh refinement strategies
can be implemented very efficiently. However, the design and
analysis of a posteriori error bounds for the VEM is a challenging task.
References~\cite{BMm2as,BeBo2017,CGPS,MRR2} are the only a posteriori
error analyses for VEM currently available in the literature.
In \cite{BMm2as}, a posteriori error bounds for the $C^{1}$-conforming VEM
for the two-dimensional Poisson problem are proposed. In
\cite{BeBo2017}  a residual-based
a posteriori error estimator for the VEM discretization of the
Poisson problem with discontinuous diffusivity coefficient has been introduced and analyzed. Moreover,
in \cite{CGPS}, a posteriori error bounds are introduced for the $C^{0}$-conforming
VEM  for the discretization of second order linear elliptic reaction-convection-diffusion
problems with non-constant coefficients in two and three dimensions.
Finally, in \cite{MRR2} 
a posteriori error analysis of a virtual
element method for the Steklov eigenvalue problem has been developed.

The aim of this paper is to introduce and analyze an $\H^{1}(\O)$-VEM
that applies to general polygonal meshes, made by possibly non-convex elements,
for the two-dimensional eigenvalue problem for the linear elasticity equations.
We begin with a variational formulation of the spectral problem relying only
on the solid displacement. Then, we propose a discretization by means of VEM,
which is based on \cite{equiv} in order to construct a proper $\mathrm{L}^2$-projection operator,
which is used to approximate the bilinear form on the right hand side
of the spectral problem.
Then, we use the so-called Babu\v ska--Osborn abstract spectral approximation
theory (see \cite{BO}) to deal with the continuous and discrete solutions operators
which appear as the solution of the continuous and discrete source
problems and whose spectra are related with the solutions of the
spectral problem. 
Under rather mild assumptions on the polygonal meshes, we establish that the
resulting VEM scheme provides a correct approximation of the spectrum and
prove optimal-order error estimates for the eigenfunctions and a double order
for the eigenvalues. The second goal of this paper is to  introduce and
analyze an a posteriori error estimator of residual type for the
virtual element approximation of the eigenvalue problem.
Since  normal fluxes of the VEM solution are not computable,
they  will be replaced in the estimators by a proper projection.
We prove that the error estimator is equivalent to the error and use the
corresponding indicator to drive an adaptive scheme.
In addition, in this work we address the issue of 
comparing the proposed a posteriori error estimator
with the standard residual estimator for a finite
element method.

The outline of this article is as follows: We introduce in
Section~\ref{SEC:STAT} the variational formulation of the
spectral problem, define a solution operator and establish
its spectral characterization. In Section~\ref{SEC:Discrete},
we introduce the virtual element discrete formulation, describe
the spectrum of a discrete solution operator and establish some
auxiliary results. In Section~\ref{SEC:approximation}, we prove
that the numerical scheme provides a correct spectral approximation
and establish optimal order error estimates for the eigenvalues
and eigenfunctions using the standard theory for compact operators.
In Section~\ref{SEC:L2}, we establish an
error estimate for the eigenfunctions in the $\mathrm{L}^{2}(\O)$-norm,
which will be useful in the a posteriori error analysis.
In Section~\ref{SEC:aposteriori1}, we define the a posteriori
error estimator and proved its reliability and efficiency.
Finally, in Section~\ref{SEC:NUMER}, we report a set of numerical
tests that allow us to assess the convergence properties of the method,
to confirm that it is not polluted with spurious modes and to check
that the experimental rates of convergence agree with the theoretical ones.
Moreover, we have also made a comparison between the proposed estimator
and the standard residual error estimator for a finite element method,

Throughout the article, $\O$ is a generic Lipschitz bounded domain of $\R^2$
with boundary $\partial \O$,
we will use standard notations for Sobolev spaces, norms and seminorms.  Finally,
we employ $\0$ to denote a generic null vector and
$C$ to denote generic constants independent of the discretization
parameters $h$, which may take different values at different occurrences.

\setcounter{equation}{0}
\section{The spectral problem}
\label{SEC:STAT}
We assume that an isotropic and linearly elastic solid occupies
a bounded and connected Lipschitz domain $\O\subset \R^{2}$.
We assume that the  boundary of the solid $\partial \O$ admits
a disjoint partition $\partial \O=\G_{D}\cup\G_{N}$,
the structure being fixed on $\G_{D}$ and free of stress on $\G_{N}$.
We denote by $\boldsymbol{\nu}$ the outward unit normal
vector to the boundary $\partial \O$.
Let us consider the eigenvalue problem for the linear
elasticity equation in $\O$ with mixed boundary conditions,
written in the variational form:

\begin{problem}
\label{P0}
Find $(\l,\bw)\in \R\times \bV:=[\H^{1}_{\G_{D}}(\O)]^{2}$, $\bw\neq \bO$, such that 
\begin{equation*}
\int_{\O}\bsi(\bw):\beps(\bv)=\l\int_{\O}\varrho\bw\cdot\bv\quad
\forall\bv\in\bV,
\end{equation*}
\end{problem}
\noindent where $\bw$ is the solid displacement and $\omega:=\sqrt{\l}$
is the corresponding vibration frequency; $\varrho$
is the density of the material, which we assume a strictly positive constant.
The constitutive equation relating the Cauchy stress tensor
$\bsi$ and the displacement field $\bw$ is given by 
$$\bsi(\bw)=\Cten\beps(\bw)\quad \text{ in }\O,$$
with $\beps(\bw):=\dfrac{1}{2}\left(\nabla \bw+(\nabla \bw)^{t}\right)$
being the standard strain tensor and $\Cten$  the elasticity operator,
which we assume given by Hooke's law, i.e., 
$$\Cten\bta:=2\mu_{S}\bta+\l_{S}\tr(\bta)\bI,$$
where $\l_{S}$ and $\mu_{S}$ are the Lam\'e coefficients,
which we assume constant.

%

We introduce the following bounded bilinear forms:
\begin{align*}
a(\bw,\bv)&:=\int_{\O}\Cten\beps(\bw):\beps(\bv),\qquad\bw,\bv\in \bV,\\
b(\bw,\bv)&:=\int_{\O}\varrho\bw\cdot\bv,\qquad\bw,\bv\in \bV.
\end{align*}
Then, the eigenvalue problem above can be rewritten as follows:
\begin{problem}
\label{P1}
Find $(\l,\bw)\in \R\times\bV$, $\bw\neq \bO$, such that 
\begin{equation*}
a(\bw,\bv)= \l b(\bw,\bv) \qquad \forall \bv\in \bV.
\end{equation*}
\end{problem}
It is easy to check (as a consequence of the Korn inequality)
that $a(\bv,\bv)\geq C\|\bv\|_{1,\O}^{2}$ for all $\bv\in \bV$.
Then, the bilinear form $a(\cdot,\cdot)$ is $\bV$-elliptic.

Next, we define the corresponding solution operator:
\begin{align*}
\bT:\bV & \longrightarrow \bV,\\
\bff & \longmapsto \bT\bff:=\bu,
\end{align*}
where $\bu\in \bV$ is the unique solution of the following source problem:
\begin{equation}\label{4}
a(\bu,\bv)=b(\bff,\bv)\qquad\forall \bv\in \bV.
\end{equation}
Thus, the linear operator $\bT$
is well defined and bounded. Notice that $(\l,\bw)\in\R\times\bV$ solves
Problem~\ref{P1} if and only if  $(\mu,\bw)$ is an eigenpair of $\bT$,
i.e, if and only if  $$\bT\bw=\mu\bw\quad \text{ with } \mu:=\dfrac{1}{\l}.$$
Moreover, it is easy to check that $\mathbf{T}$ is self-adjoint
with respect to the inner product $a(\cdot,\cdot)$ in $\bV$.

The following is an additional regularity result for the solution of
problem~\eqref{4} and consequently, for the eigenfunctions of $\bT$.

\begin{lemma}
\label{LEM:REG}
There exists $r_{\O}>0$ such that the following results hold:
\begin{itemize}
\item[(i)] for all $\bff\in[\LO]^{2}$ and for all $r\in(0,r_{\O})$,
the solution $\bu$ of problem~\eqref{4} satisfies $\bu\in[\H^{1+r_1}(\O)]^{2}$
with $r_1:=\min\{r,1\}$ and there exists $C>0$ such that 
$$
\left\|\bu\right\|_{1+r_1,\O}
\le C\left\|\bff\right\|_{0,\O}.
$$
\item[(ii)] if $\bw$ is an eigenfunction
of Problem~\ref{P1} with eigenvalue $\l$,
for all $r\in(0,r_{\O})$, $\bw\in[\HurO]^{2}$
and there exists $C>0$ (depending on $\l$) such that 
$$
\left\|\bw\right\|_{1+r,\O}
\le C\left\|\bw\right\|_{0,\O}.
$$
\end{itemize}
\end{lemma}

\begin{proof}
The proof  follows from the regularity
result for the classical elasticity problem
(cf. \cite{G2}).
\end{proof}

Hence, because of the compact inclusion $[\H^{1+r_{1}}(\O)]^{2}\hookrightarrow[\HuO]^{2}$, $\bT$ is a
compact operator. Therefore, we have the following spectral
characterization result.

\begin{theorem}
\label{CHAR_SP}
The spectrum of $\bT$  satisfies 
$\sp(\bT)=\{0\}\cup\left\{\mu_k\right\}_{k\in\N}$, where
 $\left\{\mu_k\right\}_{k\in\N}$ is a
sequence of real positive eigenvalues which converges to $0$.
The multiplicity of each eigenvalue is finite and
their corresponding eigenspaces lie in $[\HurO]^{2}$.
\end{theorem}


\setcounter{equation}{0}
\section{Virtual elements discretization}
\label{SEC:Discrete}

We begin this section, by recalling the mesh construction and
the shape regularity assumptions to introduce the discrete virtual element
space. Then, we will introduce a virtual element discretization of
Problem~\ref{P1}  and provide a spectral characterization
of the resulting discrete eigenvalue problem.
Let $\left\{\CT_h\right\}_h$ be a sequence of decompositions of $\O$ 
into polygons $E$. Let $h_E$ denote the diameter of the element $E$
and $h:=\displaystyle\max_{E\in\O}h_E$.
In what follows, we denote by $N_E$ the number of vertices of
$E$, and by $\ell$ a generic edge of $\CT_h$.

For the analysis, we will make the following
assumptions as in  \cite{BMRR}:
there exists a positive real number $C_{\T}$ such that,
for every $h$ and every $E\in \T_h$,
\begin{itemize}
\item[$\mathbf{A_1}$:]  the ratio between the shortest edge
and the diameter $h_E$ of $E$ is larger than $C_{\T}$;
\item[$\mathbf{A_2}$:]  $E\in\CT_h$ is star-shaped with
respect to every point of a  ball
of radius $C_{\T}h_E$.
\end{itemize}
Moreover, for any subset $S\subseteq\R^2$ and nonnegative
integer $k$, we indicate by $\bbP_{k}(S)$ the space of
polynomials of degree up to $k$ defined on $S$.


To continue the construction of the discrete scheme, we need some preliminary
definitions. First, we split the bilinear forms $a(\cdot,\cdot)$ and
$b(\cdot,\cdot)$, introduced in the previous section as follows:
\begin{eqnarray*}
a\left(\bu,\bv \right)&=&\sum_{E\in\CT_h}a^{E}\left(\bu,\bv\right),
\quad\text{ and }\quad b\left(\bu,\bv\right)=\sum_{E\in\CT_h}b^{E}\left(\bu,\bv\right),\quad \bu,\bv\in\bV
\end{eqnarray*}
with
\begin{eqnarray*}
\label{formasa}
a^{E}\left(\bu,\bv \right)&:=&\int_{E}\Cten\beps(\bu):\beps(\bv)\qquad\forall \bu,\bv 
\in[\H^{1}(\O)]^{2} ,\\\label{formasc}
b^{E}\left(\bu,\bv\right)&:=&\int_{E}\varrho \bu\cdot\bv\qquad\forall \bu,\bv 
\in[\H^{1}(\O)]^{2}.
\end{eqnarray*}

Now, we consider a simple polygon $E$ and, for $k\in\N$, we define 
\begin{equation*}
\label{espace-1}
\boldsymbol{\mathbb{B}}_{\partial E}:=\{\bv_{h}\in [C^0(\partial E)]^{2}: \bv_{h}|_{\ell}\in 
[\bbP_k(\ell)]^{2}\ \ \forall \ell\subset
\partial E\}.
\end{equation*}
We then consider the
following finite dimensional 
space:
\begin{equation*}
 \label{space2}
\bW_h^E:=\left\{\bv_{h}\in [\H^1(E)]^2: \Delta \bv_{h}\in[\mathbb{P}_k(E)]^2\,\,
\textrm{and}\,\,\bv_{h}|_{\partial E}\in \boldsymbol{\mathbb{B}}_{\partial E}\right\}.
 \end{equation*}
The following set of linear operators are well defined for all $\bv_h\in\bW_h^E$:
 \begin{itemize}
  \item $\mathcal{V}_E^h$ : The (vector) values of $\bv_h$ at the vertices.
  \item $\mathcal{E}_E^h$, for $k>1$ : The edge moments $\disp\int_{\ell}\bp\cdot\bv_{h}$
  for  $\bp\in [\mathbb{P}_{k-2}(\ell)]^2$ on each edge $\ell$ of $E$.
  \item $\mathcal{K}_E^h$, for $k>1$ : The internal moments $\disp\int_{E}\bp\cdot\bv_{h}$
  for  $\bp\in [\mathbb{P}_{k-2}(E)]^2$ on each element $E$.
\end{itemize}
Now we define the projector
$\bpisi: \bW_h^{E}\longrightarrow [\mathbb{P 
}_k(E)]^2\subset\bW_h^{E}$
for each $\bv_h\in\bW_h^{E}$ as the solution of
\begin{align}
\label{proje_0}
\left\{\begin{array}{ll}
&  a^{E}(\bp,\bpisi\bv_h) = a^{E}(\bp,\bv_h)
\quad \forall \bp\in [\bbP_k(E)]^2,\\\\
& \left<\left<\bp,\bpisi\bv_h\right>\right>=\left<\left<\bp,
\bv_h\right>\right>\quad\forall \bp\in ker(a^E(\cdot,\cdot)), 
\end{array}\right.
\end{align}
where for all ${\bf r}_h,{\bf s}_h\in\bW_h^{E}$,
$$
\left<\left<{\bf r}_h,{\bf s}_h\right>\right>:=\disp\dfrac{1}{N_{E}}\sum_{i=1}^{N_{E}}
{\bf r}_h(\vv_i)\cdot {\bf s}_h (\vv_i),\quad\vv_i=\text{ vertices of }E , \; 1\leq i\leq N_{E} .
$$
We note that the second equation in \eqref{proje_0} is needed for the problem to be well-posed.

Now, we introduce our local virtual space:
\begin{equation*}
 \label{space_3}
\bV_h^E:=\left\{\bv_{h}\in 
\bW_h^E: \displaystyle\int_E 
\boldsymbol{p}\cdot\bpisi\bv_{h}=\displaystyle\int_E\boldsymbol{p}\cdot\bv_{h},\quad \forall \boldsymbol{p}\in 
[\mathbb{P}_k(E)]^2/[\mathbb{P}_{k-2}(E)]^2\right\},
 \end{equation*}
where the space $[\mathbb{P}_k(E)]^2/[\mathbb{P}_{k-2}(E)]^2$ denote the polynomials
in $[\mathbb{P}_k(E)]^2$ that are $[{\mathrm L}^{2}(E)]^{2}$ orthogonal to $[\mathbb{P}_{k-2}(E)]^2$.
We observe that, since $\bV_h^E\subset \bW_h^E$, the operator $\bpisi$ is well defined on $\bV_h^E$ and computable 
only on the basis of the output values of the operators in $\mathcal{V}_E^h$,  $\mathcal{E}_E^h$ and $\mathcal{K}_E^h$.
We note that it can be proved, see \cite{equiv,BBCMMR2013,BBMRm3as2016} that the set of linear operators
$\mathcal{V}_E^h$,  $\mathcal{E}_E^h$ and $\mathcal{K}_E^h$ constitutes a set of degrees of
freedom for the local virtual space $\bV_h^E$. Moreover, it is easy to check that
$[\mathbb{P}_k(E)]^2\subset \bV_h^E$. This will guarantee the good approximation properties for the space.

Additionally, we have that the  standard $[{\mathrm L}^2(E)]^2$-projector operator
$\bpiO: \bV_h^{E}\to[\bbP_k(E)]^2$  can be  computed from the set
of degrees freedom. In fact, for all $\bv_h\in  \bV_h^E$, the function $\bpiO\bv_h\in [\bbP_k(E)]^2$ is defined by:
\begin{equation*}
\label{pi0}
\int_E\boldsymbol{p}\cdot\bpi_0^E\bv_h=\left\{\begin{array}{ll}
 \disp\int_E \boldsymbol{p}\cdot\bpisi\bv_h,&\quad\forall \boldsymbol{p}\in[\mathbb{P}_k(E)]^2/[\mathbb{P}_{k-2}(E)]^2
,\\\\
\disp \int_E\boldsymbol{p}\cdot\bv_h,&\quad\forall \bp\in [\mathbb{P}_{k-2}(E)]^2. 
\end{array}\right. 
\end{equation*}

We can now present the global virtual space:
for every decomposition $\CT_h$ of $\O$ into
simple polygons $E$.
\begin{align*}
\bV_h:=\left\{\bv_{h}\in\bV:\bv_h|_{E}\in  \bV_h^{E},\quad \forall E\in \CT_h \right\}.
\end{align*}
In agreement with the local choice of the degrees of freedom, in 
$\bV_h$ we choose the following degrees of freedom:
\begin{itemize}
\item $\mathcal{V}^h$: the (vector) values of $\bv_h$ at the vertices of $\CT_h$.
 \item $\mathcal{E}^h$, for $k>1$ : The edge moments $\disp\int_{\ell}\bp\cdot\bv_{h}\quad
 \forall \bp\in [\mathbb{P}_{k-2}(\ell)]^2$ on each edge $\ell\not\subset\G_{D}$ .
  \item $\mathcal{K}^h$, for $k>1$ : The internal moments $\disp\int_{E}\bp\cdot\bv_{h}\quad
  \forall \bp\in [\mathbb{P}_{k-2}(E)]^2$ on each element $E\in \CT_{h}$.
\end{itemize}

On the other hand, let   $S_{\beps}^E(\cdot,\cdot)$ and $S_0^E(\cdot,\cdot)$
be  symmetric
positive definite bilinear forms  chosen as to satisfy 
\begin{align}
&\disp c_0a^{E}(\bv_h,\bv_h )\leq S_{\beps}^E(\bv_h,\bv_h)\leq\disp c_1a^{E}(\bv_h,\bv_h )
\quad\forall\bv_h\in \bV_h^{E} 
\ \textrm{ with } \bpisi \bv_{h} = \boldsymbol{0},\label{stabilS} \\
&\disp \tilde{c}_0 b^E(\bv_h,\bv_h)\leq S_0^E(\bv_h,\bv_h)\leq\disp \tilde{c}_1b^{E}(\bv_h,\bv_h)
\quad\forall\bv_h\in \bV_h^{E},
\label{stabilS0}
\end{align}
for some positive constants $c_0$, $c_1$, $\tilde{c}_0$ and $\tilde{c}_1$
depending only on the constant $C_{\CT}$ that appears in assumptions $\bA_1$ and $\bA_2$.
Then, we introduce on each element $E$ the local (and computable) bilinear forms 
\begin{align}
a_h^E(\bu_h,\bv_h):=a^E(\bpisi\bu_h,\bpisi\bv_h)
+S_{\beps}^E(\bu_h-\bpisi\bu_h,\bv_h-\bpisi\bv_h)\qquad \bu_h,\bv_h\in  \bV_h^{E},\label{bilineal_ae}\\
b_h^E(\bu_h,\bv_h):=b^E(\bpiO\bu_h,\bpiO\bv_h)+S_0^E(\bu_{h}-\bpiO\bu_h,\bv_h-\bpiO\bv_h)
\qquad \bu_h,\bv_h\in  \bV_h^{E}.\label{bilineal_be}
\end{align}

Now, we define in a natural way
\begin{equation*}
\label{bilinearforms}
a_{h}(\bu_h,\bv_h)=\sum_{E\in\CT_{h}}a_h^{E}(\bu_h,\bv_h),
\qquad b_{h}(\bu_h,\bv_h):=\sum_{E\in\CT_{h}}b_{h}^{E}(\bu_h,\bv_h)\qquad \bu_h,\bv_h\in \bV_h.
\end{equation*}

The construction of $a_h^E(\cdot,\cdot)$ and $b_h^E(\cdot,\cdot)$ 
guarantees the usual \textit{consistency} and  \textit{stability}
properties of VEM, as noted in the proposition below.
Since the proof is simple and follows standard arguments
in the Virtual Element literature, it is omitted (see \cite{BBCMMR2013}).
\begin{proposition}
The local bilinear forms $a_h^E(\cdot,\cdot)$ and $b_h^E(\cdot,\cdot)$ on each element $E$ satisfy
\begin{itemize}
 \item Consistency: for all $h>0$ and for all $E\in\CT_h$ we have that
\begin{align}
\label{consis0}
a_h^E\left(\bp,\bv_h\right)&=a^E(\bp,\bv_h)\quad\forall\bp\in[\bbP_k(E)]^2,\; \forall\bv_h\in \bV_h^{E};\\
 \label{consis1}
b_h^E\left(\bp,\bv_h\right)&=b^E(\bp,\bv_h)\quad\forall\bp\in[\bbP_k(E)]^2,\; \forall\bv_h\in \bV_{h}^{E}.
 \end{align}
 \item Stability: there exist positive constants
 $\alpha_{*}$, $\alpha^{*}$, $\beta_{*}$ and $\beta^{*}$, independent of $h$ and $E$, such that
\begin{align}
\alpha_{*}a^E(\bv_h,\bv_h)&\leq a_h^E(\bv_h,\bv_h)\leq\alpha^{*}a^E(\bv_h,\bv_h)\quad\forall 
\bv_h\in \bV_h^{E},\quad \forall E\in \CT_{h},\label{stab0}\\
\beta_{*}b^E(\bv_h,\bv_h)&\leq b_h^E(\bv_h,\bv_h)\leq\beta^{*}b^E(\bv_h,\bv_h)\quad\forall 
\bv_h\in \bV_h^{E},\quad \forall E\in \CT_{h}.\label{stab2}
\end{align}
\end{itemize}
\end{proposition}

Now, we are in a position to write the virtual
element discretization of Problem~\ref{P1}.
\begin{problem}
\label{P3}
Find $(\l_h,\bw_h)\in \R\times\bV_{h}$, $\bw_h\neq \0$, such that 
\begin{equation*}\label{vp}
a_{h}(\bw_h,\bv_h)=\l_h b_h(\bw_h,\bv_h) \qquad \forall \bv_h\in\bV_{h}.
\end{equation*}
\end{problem}
We observe that by virtue of  \eqref{stab0},
the bilinear form $a_h(\cdot,\cdot)$ is bounded. Moreover, as is shown in
the following lemma, it is also uniformly elliptic.
\begin{lemma}
\label{ha-elipt-disc}
There exists a constant $\beta>0$, independent of $h$, such that
$$
a_h(\bv_h,\bv_h)
\ge\beta\left\|\bv_h\right\|_{1,\O}^{2}
\qquad\forall \bv_h\in\bV_{h}.
$$
\end{lemma}
\begin{proof}
Thanks to \eqref{stab0}, it is easy to check that
the above inequality holds with
$\beta:=\min\left\{\alpha_{*},1\right\}$.
\end{proof}

The next step is to introduce the discrete version of operator $\bT$:
\begin{align*}
\bT_h:\bV_{h} & \longrightarrow \bV_{h},\\
\bff_h & \longmapsto \bT_h\bff_h:=\bu_h,
\end{align*}
where $\bu_h\in\bV_{h}$ is the solution of the corresponding
discrete source problem:
\begin{equation}\label{T2}
a_h(\bu_h,\bv_h)=b_h(\bff_h,\bv_h)\qquad\forall \bv_h\in\bV_{h}.
\end{equation}
We deduce from Lemma \ref{ha-elipt-disc}, \eqref{stab0}--\eqref{stab2}
and the Lax-Milgram Theorem,
that the linear operator $\bT_h$ 
is well defined and bounded uniformly with respect to $h$.

Once more, as in the continuous case, $(\l_h,\bw_h)$ solves Problem~\ref{P3} if and 
only if  $(\mu_h,\bw_{h})$ is an eigenpair of $\bT_h$, i.e, if 
and only if 
$$\bT_h\bu_{h}=\mu_h\bu_{h}\quad \text{ with }\;\; \mu_h:=\dfrac{1}{\l_h}.$$
Moreover, it is easy to check that $\bT_h$ is self-adjoint with respect to 
$a_h(\cdot,\cdot)$ and $b_h(\cdot,\cdot)$. 

As a consequence, we have the following spectral characterization
of the discrete solution operator.

\begin{theorem}
\label{CHAR_SP_DISC}
The spectrum of $\bT_h$ consists of $M_h:=\dim(\bV_{h})$ eigenvalues
repeated according
to their respective multiplicities. All of them are real and positive.
\end{theorem}
\setcounter{equation}{0}
\section{Spectral approximation and error estimates}
\label{SEC:approximation}
To prove that $\bT_{h}$ provides a correct spectral approximation
of $\bT$, we will resort to the classical theory for compact operators (see \cite{BO}).
With this aim, we recall the following approximation result which is derived by interpolation
between Sobolev spaces (see for instance \cite[Theorem I.1.4]{GR}
from the analogous result for integer values of $s$.
In its turn, the result for integer values is stated
in \cite[Proposition 4.2]{BBCMMR2013} and follows from the
classical Scott-Dupont theory (see \cite{BS-2008}):
\begin{lemma}
\label{estima2}
Assume $\bA_1$ and $\bA_2$ are satisfied.
There exists a constant $C>0$,  such that for every
$\bv\in [\H^{1+t}(E)]^2$ with $0\leq t\leq k$,
there exists $\bv_{\Pi}\in [\bbP_k(E)]^2$, $k\geq 0$ such that
\begin{eqnarray*}
\|\bv-\bv_{\Pi}\|_{0,E}+h_{E}\arrowvert \bv-\bv_{\Pi}\arrowvert_{1,E}\leq C h_E^{1+t}|\bv|_{1+t,E}.
\end{eqnarray*}
\end{lemma}

The classical theory for compact operators, is based on the
convergence in norm of $\bT_{h}$ to $\bT$ as $h\rightarrow 0$.
However, the operator $\bT_{h}$ is not well defined for any
$\bff\in\bV$, since the definition of bilinear form $S_{0}^{E}(\cdot,\cdot)$
in \eqref{stabilS0} needs the degrees of freedom and in particular
the pointwise values of $\bff$.
To circumvent this drawback, we introduce the projector
$\bP_h:\;[\LO]^{2}\longrightarrow\bV_{h}\hookrightarrow\bV$ with range $\bV_h$,
which is defined by the relation 
\begin{equation}
\label{prop}b(\bP_h\bu-\bu,\bv_h)=0
\qquad\forall\bv_h\in\bV_h.
\end{equation}
In our case, the bilinear form $b(\cdot,\cdot)$ correspond
to the $\LO$ inner product. Thus, $\left\|\bP_h\bu\right\|_{0,\O}\leq\left\|\bu\right\|_{0,\O}$.
Moreover,
\begin{equation}\label{errorpro}
\Vert\bu-\bP_h\bu\Vert_{0,\O}=\inf_{\bv_h\in\bV_{h}}\Vert\bu-\bv_h\Vert_{0,\O}.
\end{equation}

For the analysis we introduce the following broken seminorm:
\begin{equation}
\vert \bv\vert_{1,h,\O}^2:=\sum_{E\in\CT_h}\vert \bv\vert_{1,E}^2,
\end{equation}
which is well defined for every $\bv\in[\LO]^{2}$
such that $\bv|_{E}\in[\H^1(E)]^{2}$ for all polygon $E\in\CT_h$.

Now, we define
$\widehat{\bT}_{h}:=\bT_{h}\bP_h:\;\bV\longrightarrow\bV_{h}$. Notice
that $\sp(\widehat{\bT}_h)=\sp(\bT_h)\cup\{0\}$ and the eigenfunctions
of $\widehat{\bT}_h$ and $\bT_h$ coincide.
Furthermore, we have the following result.

\begin{lemma}
\label{lemcotste}
There exists $C>0$ such that, for all $\bff\in\bV$, if $\bu:=\mathbf{T}\bff$
and $\bu_h:=\widehat{\bT}_{h}\bff=\mathbf{T}_h\bP_h\bff$, then
\begin{align*}
\|(\mathbf{T}-\widehat{\bT}_h)\bff\|_{1,\O}& \leq C\left(h\left(||\bff-\bff_{I}\|_{0,\O}+\|\bff-\bff_{\pi}\|_{0,\O}\right)
+\|\bu-\bu_{I}\|_{1,\O}+|\bu-\bu_{\pi}|_{1,h,\O}\right),
\end{align*}
for all $\bu_{I},\bff_{I}\in \bV_{h}$, for all $\bu_{\pi}\in [\LO]^{2}$
such that $\bu_{\pi}|_{E}\in [\bbP_k(E)]^2$ $\forall E\in\T_{h}$ and for all $\bff_{\pi}\in [\LO]^{2}$
such that $\bff_{\pi}|_{E}\in [\bbP_k(E)]^2$ $\forall E\in\T_{h}$.
\end{lemma}

\begin{proof}
Let $\bff\in\bV $, for $\bu_{I}\in \bV_{h}$ we have that
\begin{equation}
\label{triangle}
 \|(\mathbf{T}-\widehat{\bT}_h)\bff\|_{1,\O}\leq
 \|\bu-\bu_{I}\|_{1,\O}+\|\bu_{I}-\bu_h\|_{1,\O}.
\end{equation}
Now, if we define $\bv_h:=\bu_h-\bu_{I}\in \bV_{h}$,
thanks to Lemma~\ref{ha-elipt-disc}, the definition of $a_h^{E}(\cdot,\cdot)$
(cf \eqref{bilineal_ae}) and those of $\mathbf{T}$ and $\bT_h$, we have 
\begin{align*}
\beta\left\|\bv_h\right\|_{1,\O}^2
&\leq a_h(\bv_h,\bv_h)
=a_h(\bu_h,\bv_h)-a_h(\bu_{I},\bv_h)=b_h(\bP_h\bff,\bw_h)-\sum_{E\in \T_{h}}a_{h}^{E}(\bu_{I},\bv_{h})\\
&=\underbrace{b_h(\bP_h\bff,\bv_h)-b(\bff,\bv_{h})}_{T_{1}}-\underbrace{\sum_{E\in \T_{h}}\left[a_{h}^{E}(\bu_{I}-\bu_{\pi},\bv_{h})+a^{E}(\bu_{\pi}-\bu,\bv_{h})\right]}_{T_{2}},
\end{align*}
where we have used the \textit{consistency} property~\eqref{consis0}
to derive the last equality. 
We now bound each term $T_{i},i=1,2$, with a constant $C>0$. 

The term $T_{1}$ can be bounded as follows:
Let $\bv_{h}^{\pi}\in [\bbP_k(E)]^2$ such that Lemma~\ref{estima2}
holds true, then by \eqref{prop}, we have 
 \begin{align*}
T_{1}&=b_h(\bP_h\bff,\bv_h)-b(\bP_h\bff,\bv_{h})= \sum_{E\in \T_{h} }\left[b_h^{E}(\bP_h\bff,\bv_h-\bv_{h}^{\pi})-b^{E}(\bP_h\bff,\bv_{h}-\bv_{h}^{\pi})\right]\\
& =\sum_{E\in \T_{h} }\left[b_h^{E}(\bP_h\bff-\bff_{\pi},\bv_h-\bv_{h}^{\pi})-b^{E}(\bP_h\bff-\bff_{\pi},\bv_{h}-\bv_{h}^{\pi})\right]\\
&\leq C\left(\sum_{E\in \T_{h} }\|\bP_h\bff-\bff_{\pi}\|_{0,E}^{2}\right)^{1/2}\left(\sum_{E\in \T_{h} }\|\bv_h-\bv_{h}^{\pi}\|_{0,E}^{2}\right)^{1/2}\\
&\leq C\|\bP_h\bff-\bff_{\pi}\|_{0,\O}\left(\sum_{E\in \T_{h} }h_{E}^2|\bv_h|_{1,E}^2\right)^{1/2}\\
&\leq Ch\left(\|\bP_h\bff-\bff\|_{0,\O}+\|\bff-\bff_{\pi}\|_{0,\O}\right)||\bv_{h}||_{1,\O}\\
&\leq Ch\left(||\bff-\bff_{I}\|_{0,\O}+\|\bff-\bff_{\pi}\|_{0,\O}\right)||\bv_{h}||_{1,\O}
\end{align*}
where we have used the definitions of $b_{h}(\cdot,\cdot)$
and $b(\cdot,\cdot)$, the {\it consistency} and {\it stability} properties
\eqref{consis1} and \eqref{stab2}, respectively,
together with Cauchy-Schwarz inequality, Lemma~\ref{estima2} and \eqref{errorpro}.

To bound $T_{2}$, we first use the  {\it stability} property~\eqref{stab0},
Cauchy-Schwarz inequality again and  adding and subtracting  $\bu$ to obtain 
\begin{align*}
T_{2}&\leq C\sum_{E\in \T_{h}}\left(|\bu-\bu_{I}|_{1,E}+2|\bu-\bu_{\pi}|_{1,E}\right)|\bv_{h}|_{1,E}.
\end{align*}
Therefore, by combining the above bounds, we obtain
\begin{align*}
\beta\left\|\bv_h\right\|_{1,\O}&\leq   C\left(h||\bff-\bff_{I}\|_{0,\O}+h\|\bff-\bff_{\pi}\|_{0,\O}+\|\bu-\bu_{I}\|_{1,\O}+|\bu-\bu_{\pi}|_{1,h,\O}\right).
\end{align*}
Hence, the proof follows from the above estimate  and \eqref{triangle}.
 \end{proof}

The next step is to find appropriate term $\bu_{I}$ that can
be used in the above lemma. Thus, we have the following result. 
\begin{lemma}
\label{estima4}
Assume $\bA_1$ and $\bA_2$ are satisfied.
Then, for every $\bv\in [\H^{1+t}(E)]^2$ with $0\leq t\leq k$,  there exists  
$\bv_{I}\in \bV_{h}$ and a constant $C > 0$, such that
 \begin{eqnarray*}
\|\bv-\bv_{I}\|_{0,E}+h_{E}\arrowvert \bv-\bv_{I}\arrowvert_{1,E}\leq C h_E^{1+t}|\bv|_{1+t,E}.
\end{eqnarray*}
\end{lemma}
\begin{proof}
The proof is identical to that of Theorem 11 from \cite{CGPS}
(in the 2D case), but using the following estimate
$$\|\bv-\bv_{c}\|_{0,T}+h|\bv-\bv_{c}|_{1,T}\leq \widehat{C}_{Clem}h^{1+t}\|\bv\|_{1+t,\widetilde{T}},$$
instead of estimate (4.2) of Theorem~11 from \cite{CGPS},
where $\bv_{c}$ is an adequate Cl\'ement interpolant of degree $k$ of $\bv$
(see \cite[Proposition 4.2]{MRR2015}).
\end{proof}
Now, we are in a position to conclude  that $\widehat{\bT}_h$ converges in norm to $\mathbf{T}$
as $h$ goes to zero.

\begin{corollary}
\label{cotaT}
There exist $C>0$ independent of $h$ and $r_1>0$ (as in Lemma~\ref{LEM:REG}(i)), such that
\begin{equation*}
\|(\mathbf{T}-\widehat{\bT}_h)\bff\|_{1,\O}\leq Ch^{r_1}\|\bff\|_{1,\O}\quad \forall \bff\in\bV.
\end{equation*}
\end{corollary}

\begin{proof}
The result follows from Lemmas~\ref{estima2}--\ref{estima4}  and Lemma~\ref{LEM:REG}.
\end{proof}
As a direct consequence of Corollary~\ref{cotaT}, standard results about
spectral approximation (see \cite{K}, for instance) show that isolated
parts of $\sp(\bT)$ are approximated by isolated parts of $\sp(\widehat{\bT}_h)$
and therefore by $\sp(\bT_h)$. More precisely,
let $\mu\ne0$ be an isolated eigenvalue of $\bT$ with
multiplicity $m$ and let $\CE$ be its associated eigenspace. Then, there
exist $m$ eigenvalues $\mu^{(1)}_h,\dots,\mu^{(m)}_h$ of $\bT_h$ (repeated
according to their respective multiplicities) which converge to $\mu$.
Let $\CE_h$ be the direct sum of their corresponding associated
eigenspaces.

We recall the definition of the \textit{gap} $\hdel$ between two closed
subspaces $\boldsymbol{\mathcal{X}}$ and $\boldsymbol{\mathcal{Y}}$ of $\bV$:
$$
\hdel(\boldsymbol{\mathcal{X}},\boldsymbol{\mathcal{Y}})
:=\max\left\{\delta(\boldsymbol{\mathcal{X}},\boldsymbol{\mathcal{Y}}),\delta(\boldsymbol{\mathcal{Y}},
\boldsymbol{\mathcal{X}})\right\},$$
where
$$
\delta(\boldsymbol{\mathcal{X}},\boldsymbol{\mathcal{Y}})
:=\sup_{\mathbf{x}\in\boldsymbol{\mathcal{X}}:\ 
\left\|\mathbf{x}\right\|_ {1,\O}=1}\delta(\mathbf{x},\boldsymbol{\mathcal{Y}}),
\quad\text{with }\delta(\mathbf{x},\boldsymbol{\mathcal{Y}}):=\inf_{\mathbf{y}\in\boldsymbol{\mathcal{Y}}}\|\mathbf{x}-\mathbf{y}\|_ {1,\O}.$$

The following error estimates for the approximation of eigenvalues and
eigenfunctions hold true.

\begin{theorem}
\label{gap}
There exists a strictly positive constant $C$ such that
\begin{align*}
\hdel(\CE,\CE_h) 
& \le C\gamma_{h},
\\
\left|\mu-\mu_h^{(i)}\right|
& \le C\gamma_{h},\qquad i=1,\dots,m,
\end{align*}
where
$$
\gamma_{h}:=\sup_{\bff\in\CE:\ \left\|\bff\right\|_{1,\O}=1}
\left\|(\mathbf{T}-\widehat{\bT}_h)\bff\right\|_{1,\O}.
$$
\end{theorem}

\begin{proof}
As a consequence of Corollary~\ref{cotaT}, $\widehat{\bT}_h$ converges in norm to $\mathbf{T}$
as $h$ goes to zero. Then, the proof follows as a direct consequence of
Theorems~7.1 and 7.3 from \cite{BO}.
\end{proof}

The theorem above yields error estimates depending on $\gamma_{h}$.
The next step is to show an optimal-order estimate for this term.
\begin{theorem}
\label{cotaT2}
 There exist  $r>0$ and $C>0$, independent of $h$, such that
\begin{equation*}
\|(\mathbf{T}-\widehat{\bT}_h)\bff\|_{1,\O}\leq Ch^{\min\{r,k\}}\|\bff\|_{1,\O} \qquad\forall \bff\in \CE,
\end{equation*}
and consequently,
$\gamma_{h}\leq Ch^{\min\{r,k\}}.$
\end{theorem}
\begin{proof}
The proof is identical to that of  Corollary~\ref{cotaT},
but using now the additional
regularity from Lemma~\ref{LEM:REG}(ii).
\end{proof}
The error estimate for the eigenvalue $\mu$ of $\bT$ leads to an analogous
estimate for the approximation of the eigenvalue $\l=1/\mu $ of Problem \ref{P1} by means $\mu$
of the discrete eigenvalues $\l_{h}^{i}=1/\mu_{h}^{i}$, $1\leq i\leq m$, of Problem \ref{P3}. However, the 
order of convergence in Theorem \ref{gap} is not optimal for $\mu$ and,
hence, not optimal for $\l$ either. Our next goal is to improve this order.

\begin{theorem}\label{convibration}
There exists $C>0$ independent of $h$ such that
\begin{equation*}
\big|\l-\l_h^{(i)}\big|
\le Ch^{2\min\{r,k\}},\qquad i=1,\ldots,m.
\end{equation*}
\end{theorem}
\begin{proof}
 Let $\bw_h\in \CE_{h}$ be an eigenfunction corresponding to one
 of the eigenvalues $\l_{h}^{(i)}$ $(i=1,\ldots,m)$ with $\left\|\bw_h\right\|_{1,\O}=1$. According to
Theorem \ref{gap},  there exists $(\l,\bw)$ eigenpair  of Problem~\ref{P1}
such that 
\begin{equation}\label{fin4}
\left\|\bw-\bw_h\right\|_{1,\O}\le C\gamma_h.
\end{equation}
From the symmetry of the bilinear forms and the facts that 
$a( \bw,\bv)=\l b(\bw,\bv)$ for all $\bv\in\bV$ (cf. Problem~\ref{P1}) and 
$a_{h}( \bw_{h},\bv_{h})=\l^{(i)}_{h} b(\bw_{h},\bv_{h})$ for all $\bv_h\in\bV_{h}$ (cf.
Problem~\ref{P3}), we have
\begin{align*}
a(\bw-\bw_{h},\bw-\bw_{h})-\l b(\bw-\bw_{h},\bw-\bw_{h})&=a(\bw_{h},\bw_{h})-\l b(\bw_{h},\bw_{h})\\
&\hspace{-3cm}=a(\bw_{h},\bw_{h})-a_{h}(\bw_{h},\bw_{h})+\l^{(i)}_{h}\left[b_{h}(\bw_{h},\bw_{h})- b(\bw_{h},\bw_{h})\right]\\
&\hspace{-3cm}+(\l^{(i)}_{h}-\l) b(\bw_{h},\bw_{h}),
\end{align*}
thus, we obtain the following identity:
\begin{align}
\nonumber
(\l^{(i)}_{h}-\l) b(\bw_{h},\bw_{h})&=a(\bw-\bw_{h},\bw-\bw_{h})-\l b(\bw-\bw_{h},\bw-\bw_{h})\\
&+a_{h}(\bw_{h},\bw_{h})-a(\bw_{h},\bw_{h})+\l^{(i)}_{h}\left[b(\bw_{h},\bw_{h})- b_{h}(\bw_{h},\bw_{h})\right].\label{45}
\end{align}
The next step is to bound each term on the right hand side above.
The first and the second ones are easily bounded using the Cauchy-Schwarz
inequality and \eqref{fin4}:
\begin{align}\label{sinnum}
\left|a(\bw-\bw_{h},\bw-\bw_{h})-\l b(\bw-\bw_{h},\bw-\bw_{h})\right| &\leq C\|\bw-\bw_{h}\|_{1,\O}^{2} \leq C\gamma_h^2.
\end{align}
For the third term, let  $\bw_{\pi}\in [\LO]^{2}$
such that $\bw_{\pi}|_{E}\in [\bbP_k(E)]^2$ $\forall E\in\T_{h}$.
From the definition of $a_{h}^{E}(\cdot,\cdot)$ (cf \eqref{bilineal_ae}),
adding and subtracting $\bw_{\pi}$ and using the \textit{consistency} property (cf \eqref{consis0})
we obtain
\begin{align*}
\vert a_{h}(\bw_{h},\bw_{h})-a(\bw_{h},\bw_{h})\vert&=\sum_{E\in\T_{h}}\left[a_{h}^{E}(\bw_{h},\bw_{h})-a^{E}(\bw_{h},\bw_{h})\right]\\
&= \sum_{E\in\T_{h}}\left[a_{h}^{E}(\bw_{h}-\bw_{\pi},\bw_{h}-\bw_{\pi})+a^{E}(\bw_{h}-\bw_{\pi},\bw_{h}-\bw_{\pi})\right]\\
&\leq C\sum_{E\in\T_{h}}\left| \bw_{h}-\bw_{\pi}\right|_{1,E}^{2}\leq C\sum_{E\in\T_{h}}\left(\left| \bw-\bw_{h}\right|_{1,E}^{2}+\left| \bw-\bw_{\pi}\right|_{1,E}^{2}\right).
\end{align*}
Then, from the last  inequality, Lemma \ref{estima2} and   \eqref{fin4},we obtain
\begin{align}
\label{doblea}
\left|a_{h}(\bw_{h},\bw_{h})-a(\bw_{h},\bw_{h})\right|\leq  C\gamma_h^2.
\end{align}

For the fourth term, repeating similar arguments to the previous case,
but using the \textit{consistency} property (cf \eqref{consis1}) we have
\begin{align*}
\left|b_{h}(\bw_{h},\bw_{h})-b(\bw_{h},\bw_{h})\right|&\leq C\sum_{E\in\T_{h}}\left\| \bw_{h}-\bw_{\pi}\right\|_{0,E}^{2}\leq C\sum_{E\in\T_{h}}\left(\left\| \bw-\bw_{h}\right\|_{0,E}^{2}+\left\| \bw-\bw_{\pi}\right\|_{0,E}^{2}\right).
\end{align*}
Then, from the last  inequality, Lemma~\ref{estima2} and \eqref{fin4}, we have
\begin{align}
\label{dobleb}
\left|b_{h}(\bw_{h},\bw_{h})-b(\bw_{h},\bw_{h})\right|&\leq C\gamma_h^2.
\end{align}
On the other hand, from the Korn's inequality and Lemma~\ref{ha-elipt-disc},
together with the fact that $\l_h^{(i)}\to\l$ as $h$ goes to zero, we have that
\begin{equation}
\label{439}
b_h(\bw_h,\bw_h)=\frac{a_{h}(\bw_{h},\bw_{h})}{\l_h^{(i)}}\geq
C\dfrac{\| \bw_h\|_{1,\O}^2}{\l_h^{(i)}}=\tilde{C}>0.
\end{equation}


Therefore, the theorem follows from \eqref{45}--\eqref{439}
and the fact that $\gamma_{h}\leq Ch^{\min\{r,k\}}$.
\end{proof}

\begin{remark}
\label{obserbapos}
The above theorem establishes that the resulting
discrete scheme provides a double order estimates
for the eigenvalues. However, we can also conclude
the following estimate which will be useful in the a posteriori
error analysis.
\begin{align}\label{cotavalor}
\left|\l-\l^{(i)}_{h}\right|&\leq C\left[\|\bw-\bw_{h}\|_{1,\O}^{2}+
\sum_{E\in\T_{h}}\left(\left\| \bw-\bpiO\bw_{h}\right\|_{0,E}^{2}
+\left|\bw-\bpisi\bw_{h}\right|_{1,E}^{2}\right)\right]\qquad i=1,\ldots,m.
\end{align}
In fact, repeating the arguments used in the proof of the above theorem (see \eqref{45}) we have 
\begin{align}\label{lamdass}
|\l-\l_h^{(i)}|&\leq C\left[\|\bw-\bw_h\|_{1,\O}^2+\vert a_h(\bw_h,\bw_h)-a(\bw_h,\bw_h)\vert+\vert b(\bw_{h},\bw_{h})- b_{h}(\bw_{h},\bw_{h})\vert\right].
\end{align}
Then, for the second and third terms on the right
hand side of \eqref{lamdass}, we use the definition
of $a_{h}^{E}(\cdot,\cdot)$ (cf \eqref{bilineal_ae}),
adding and subtracting $\bpisi\bw_{h}\in[\bbP_k(E)]^2$ $\forall E\in\T_{h}$
and using the \textit{consistency} property (cf \eqref{consis0}) we have
\begin{align*}
\vert a_h(\bw_h,\bw_h)-a(\bw_h,\bw_h)\vert&\leq C\sum_{E\in\CT_h}|\bw_h-\bpisi \bw_h|_{1,E}^{2}\leq C\sum_{E\in\CT_h}\left(|\bw- \bw_h|_{1,E}^{2}+|\bw-\bpisi \bw_h|_{1,E}^{2}\right)
\end{align*}
For the fourth and fifth terms on the right hand side of
\eqref{lamdass}, we use the definition of $b_{h}^{E}(\cdot,\cdot)$
(cf. \eqref{bilineal_be}), adding and subtracting $\bpiO\bw_{h}\in[\bbP_k(E)]^2$ $\forall E\in\T_{h}$
and using \textit{consistency} property (cf \eqref{consis1}) we obtain
\begin{align*}
\vert b(\bw_{h},\bw_{h})- b_{h}(\bw_{h},\bw_{h})\vert \leq C\sum_{E\in\CT_h}\left(\|\bw-\bw_h\|_{0,E}^{2}
+\|\bw-\bpiO \bw_h\|_{0,E}^{2}\right).
\end{align*}
Thus, \eqref{cotavalor} follows from the previous inequalities. 
\end{remark}

\section{Error estimates for the eigenfunctions in the $[\LO]^{2}$-norm}
\label{SEC:L2}

Our next goal is to derive an error estimate for
the eigenfunctions in the $[\LO]^{2}$-norm.
The main result of this section is the following bound.

\begin{theorem}
\label{asintotico}
There exists $C>0$ independent of $h$ such that
\begin{equation}
\label{meta}
\left\|\bw-\bw_h\right\|_{0,\O}\leq C h^{r_1}\left\{\|\bw-\bw_{h}\|_{1,\O}
+\left[\sum_{E\in\T_{h}}\left(\left\| \bw-\bpiO\bw_{h}\right\|_{0,E}^{2}
+\left| \bw-\bpisi\bw_{h}\right|_{1,E}^{2}\right)\right]^{1/2}\right\}.
\end{equation}
\end{theorem}

The proof of the above result will follow by combining Lemmas~\ref{lasttheo},
\ref{convnormnew} and \ref{erroga} shown in the sequel.


\begin{lemma}
\label{lasttheo}
There exist $C>0$ and $r_1>0$ (as in Lemma~\ref{LEM:REG}(i))
such that, for all $\bff\in\CE$, if $\bu:=\mathbf{T}\bff$
and $\bu_h:=\widehat{\bT}_{h}\bff=\mathbf{T}_h\bP_h\bff$, then
$$\|\bu-\bu_{h}\|_{0,\O}\leq C h^{r_1}\left\{\|\bu-\bu_{h}\|_{1,\O}+\left[\sum_{E\in\T_{h}}\left(\left\| \bu-\bpiO\bu_{h}\right\|_{0,E}^{2}+\left| \bu-\bpisi\bu_{h}\right|_{1,E}^{2}\right)\right]^{1/2}\right\}.$$
\end{lemma}
\begin{proof}
Let $\bv\in \bV$ the unique solution of the following problem:
\begin{equation*}
a(\bv,\bta)=b(\bu-\bu_{h},\bta)\qquad\forall \bta\in \bV.
\end{equation*}
Therefore, $\bv=\bT(\bu-\bu_{h})$, so that according to Lemma~\ref{LEM:REG}(i),
there exists $r_1>0$ such that $\bv\in[\H^{1+r_1}(\O)]^2$  and 
\begin{equation}
\label{L2error}
\|\bv\|_{1+r_1,\O}\leq C\|\bu-\bu_{h}\|_{0,\O} \quad\text{ with } C=C(\O,\mu_{S},\l_{S}).
\end{equation} 
Let $\bv_{I}\in \bV_{h}$ such that the estimate of Lemma~\ref{estima4} holds true.
Then, by simple manipulations, we have that
\begin{align}
\nonumber
\|\bu-\bu_{h}\|_{0,\O}^{2}&=a(\bu-\bu_{h},\bv-\bv_{I})+a(\bu-\bu_{h},\bv_{I})\\ \nonumber
&\leq C h^{r_1}|\bu-\bu_{h}|_{1,\O}|\bv|_{1+r_1,\O}+a(\bu-\bu_{h},\bv_{I})\\ 
&\leq C h^{r_1}\|\bu-\bu_{h}\|_{1,\O}\|\bu-\bu_{h}\|_{0,\O}+a(\bu-\bu_{h},\bv_{I}).
\label{l2estimate}
\end{align}

For the second term on the right hand side above, we have
the following equality
\begin{align}
\label{do}
a(\bu-\bu_{h},\bv_{I})&=\underbrace{a_{h}(\bu_{h},\bv_{I})-a(\bu_{h},\bv_{I})}_{B_1}
+\underbrace{b(\bff,\bv_{I})-b_{h}(\bP_h\bff,\bv_{I})}_{B_{2}},
\end{align}
where we have used \eqref{4}, added and subtracted $b_{h}(\bP_h\bff,\bv_{I})$
and \eqref{T2}.

To bound the term $B_{1}$, we consider $\bv_{\pi}\in [\LO]^{2}$
such that $\bv_{\pi}|_{E}\in[\bbP_k(E)]^2$ $\forall E\in\T_{h}$
and estimate of Lemma~\ref{estima2} holds true.
Then, using the  \textit{consistency} property
(cf \eqref{consis0}) twice and the \textit{stability}
property  (cf \eqref{stab0}), we obtain
\begin{align*}
B_{1}&=\sum_{E\in \CT_{h}}\left[a_{h}^{E}(\bu_{h},\bv_{I})-a^{E}(\bu_{h},\bv_{I})\right]
= \sum_{E\in\T_{h}}\left[a_{h}^{E}(\bu_{h}-\bpisi\bu_{h},\bv_{I}-\bv_{\Pi})+a^{E}(\bu_{h}-\bpisi\bu_{h},\bv_{I}-\bv_{\pi})\right]\\
&\leq C\left(\sum_{E\in\T_{h}}\left| \bu_{h}-\bpisi\bu_{h}\right|_{1,E}^{2}\right)^{1/2}\left(\sum_{E\in\T_{h}}\left| \bv_{I}-\bv_{\pi}\right|_{1,E}^{2}\right)^{1/2}\\
&\leq C\left[\sum_{E\in\T_{h}}\left(\left| \bu-\bu_{h}\right|_{1,E}^{2}
+\left| \bu-\bpisi\bu_{h}\right|_{1,E}^{2}\right)\right]^{1/2}\left[\sum_{E\in\T_{h}}\left(\left| \bv-\bv_{I}\right|_{1,E}^{2}
+\left| \bv-\bv_{\pi}\right|_{1,E}^{2}\right)\right]^{1/2}\\
&\leq C h^{r_1}\left[\left\| \bu-\bu_{h}\right\|_{1,\O}
+\left(\sum_{E\in\T_{h}}\left| \bu-\bpisi\bu_{h}\right|_{1,E}^{2}\right)^{1/2}\right]|\bv|_{1+r_1,\O},
\end{align*}
where for the last inequality, we have used Lemmas~\ref{estima2}
and \ref{estima4}. Then, from \eqref{L2error}, we obtain
\begin{align}
\label{l2a}
B_{1}\leq Ch^{r_1}\left[\left\| \bu-\bu_{h}\right\|_{1,\O}
+\left(\sum_{E\in\T_{h}}\left| \bu-\bpisi\bu_{h}\right|_{1,E}^{2}\right)^{1/2}\right]\|\bu-\bu_{h}\|_{0,\O}.
\end{align}
For the term $B_{2}$, we use the fact that $\bff\in \CE$,
$\bu=\mathbf{T}\bff=\mu\bff$, \eqref{prop}, the  \textit{consistency} property
\eqref{consis1} twice and the \textit{stability} property (cf \eqref{stab2}), to obtain
\begin{align*}
B_{2}&=b(\bP_h\bff,\bv_{I})-b_{h}(\bP_h\bff,\bv_{I})=\sum_{E\in \CT_{h}}\left[b^{E}(\bP_h\bff,\bv_{I})-b_{h}^{E}(\bP_h\bff,\bv_{I})\right]\\
&= \sum_{E\in\T_{h}}\left[b^{E}(\bP_h\bff-\mu^{-1}\bpiO\bu_{h},\bv_{I}-\bv_{\pi})-b_{h}^{E}(\bP_h\bff-\mu^{-1}\bpiO\bu_{h},\bv_{I}-\bv_{\pi})\right]\\
&\leq C\left(\sum_{E\in\T_{h}}\left\| \bP_h\bff-\mu^{-1}\bpiO\bu_{h}\right\|_{0,E}^{2}\right)^{1/2}\left(\sum_{E\in\T_{h}}\left\| \bv_{I}-\bv_{\pi}\right\|_{0,E}^{2}\right)^{1/2}\\
&\leq Ch^{1+r_1}\left(\sum_{E\in\T_{h}}\left\| \bP_h\bff
-\mu^{-1}\bpiO\bu_{h}\right\|_{0,E}^{2}\right)^{1/2}\|\bu-\bu_{h}\|_{0,\O},
\end{align*}
where for the last inequality,
we have used Lemmas \ref{estima2} and \ref{estima4} together with \eqref{L2error}.
Now, we have that  
\begin{align*}
\left(\sum_{E\in\T_{h}}\left\| \bP_h\bff-\mu^{-1}\bpiO\bu_{h}\right\|_{0,E}^{2}\right)^{1/2}&=\left|\mu^{-1}\right|\left(\sum_{E\in\T_{h}}\left\| \bP_h\bu-\bpiO\bu_{h}\right\|_{0,E}^{2}\right)^{1/2}\\
&\leq
\left|\mu^{-1}\right|\left[\sum_{E\in\T_{h}}\left(\left\| \bP_h\bu-\bu\right\|_{0,E}^{2}
+\left\| \bu-\bpiO\bu_{h}\right\|_{0,E}^{2}\right)\right]^{1/2}\\
&\leq C\left[\left\| \bu-\bP_h\bu\right\|_{0,\O}+\left(\sum_{E\in\T_{h}}\left\| \bu-\bpiO\bu_{h}\right\|_{0,E}^{2}\right)^{1/2}\right]\\
&\leq C\left[\left\| \bu-\bu_{h}\right\|_{0,\O}+\left(\sum_{E\in\T_{h}}\left\| \bu-\bpiO\bu_{h}\right\|_{0,E}^{2}\right)^{1/2}\right],
\end{align*}
where we have used the fact that $\bff\in \CE$,
$\bu=\mathbf{T}\bff=\mu\bff$ and the stability property of $\bP_{h}$.
Thus, we obtain
\begin{align}
\label{bl2}
B_{2}\leq Ch^{1+r_1}\left[\left\| \bu-\bu_{h}\right\|_{1,\O}
+\left(\sum_{E\in\T_{h}}\left\| \bu-\bpiO\bu_{h}\right\|_{0,E}^{2}\right)^{1/2}\right]\|\bu-\bu_{h}\|_{0,\O}.
\end{align}

Finally, combining \eqref{do}--\eqref{bl2}, with \eqref{l2estimate}
allow us to conclude the proof. 
\end{proof}

The next step is to define a solution operator on the space $[\LO]^{2}$:
\begin{align*}
\widetilde{\mathbf{T}}:\ [\LO]^{2} & \longrightarrow [\LO]^{2},
\\
\widetilde{\bff} & \longmapsto \widetilde{\mathbf{T}}\widetilde{\bff}:=\widetilde{\bu},
\end{align*}
where $\widetilde{\bu}\in\bV$ is the unique solution of the following problem:
\begin{equation}
\label{T14}
a(\widetilde{\bu},\bv)=b(\widetilde{\bff}, \bv)\qquad\forall \bv\in\bV.
\end{equation}
It is easy to check that the operator $\widetilde{\mathbf{T}}$
is compact and self-adjoint. Moreover, the spectra of $\mathbf{T}$
and  $\widetilde{\mathbf{T}}$ coincide.

Now, we will establish the convergence  of $\widehat{\bT}_{h}$ to
 $\widetilde{\mathbf{T}}$.

\begin{lemma}
\label{convnormnew}
There exist $C>0$ and $r_1>0$ (as in Lemma~\ref{LEM:REG}(i)) such that
$$
\big\|(\widetilde{\mathbf{T}}-\widehat{\bT}_{h})\bff\big\|_{0,\O}
\leq Ch^{r_1}\big\|\bff\big\|_{0,\O}\quad\forall \bff\in[\LO]^2.
$$
\end{lemma}
\begin{proof}
Given $\bff\in[\LO]^{2}$, let $\bu\in\bV$ and $\bu_h\in\bV_{h}$ be the solutions of
problems \eqref{T14} and \eqref{T2}, respectively. Hence,
$\bu=\widetilde{\mathbf{T}}\bff$ and $\bu_h=\widehat{\bT}_{h}\bff$.
The arguments used in the proof of Lemma \ref{lemcotste} can be repeated,  however to bound the term $T_{1}$, we use
\begin{align*}
\label{513}
T_{1}&=b_h(\bP_h\bff,\bv_h)-b(\bP_h\bff,\bv_{h})= \sum_{E\in \T_{h} }\left[b_h^{E}(\bP_h\bff,\bv_h-\bv_{h}^{\pi})-b^{E}(\bP_h\bff,\bv_{h}-\bv_{h}^{\pi})\right]\\
&\leq C\left(\sum_{E\in \T_{h} }\|\bP_h\bff\|_{0,E}^{2}\right)^{1/2}\left(\sum_{E\in \T_{h} }\|\bv_h-\bv_{h}^{\pi}\|_{0,E}^{2}\right)^{1/2}\\
&\leq C\|\bP_h\bff\|_{0,\O}\left(\sum_{E\in \T_{h} }h_{E}^2|\bv_h|_{1,E}^2\right)^{1/2}\leq
Ch\|\bff\|_{0,\O}||\bv_{h}||_{1,\O}.
\end{align*}
Therefore, in this case, we obtain
$$\|\bu-\bu_{h}\|_{1,\O}\leq C\left(h\|\bff\|_{0,\O}+\|\bu-\bu_{I}\|_{1,\O}+|\bu-\bu_{\pi}|_{1,h,\O}\right)$$ 
where $\bu_{I}$ and $\bu_{\pi}$ are defined as in that lemma. 
Thus, the result follows from
Lemmas~\ref{estima2}--\ref{estima4} and Lemma~\ref{LEM:REG}.
\end{proof}

As a consequence of this lemma, a spectral convergence result analogous
to Theorem~\ref{gap} holds for $\widehat{\bT}_{h}$ and $\widetilde{\mathbf{T}}$.
Moreover, we are in a position to establish the following estimate.

\begin{lemma}
\label{erroga}
Let $\bw_h$ be an eigenfunction of $\widehat{\bT}_{h}$ associated with the eigenvalue
$\mu_h^{(i)}$, $1\le i\le m$, with $\left\|\bw_h\right\|_{0,\O}=1$. Then,
there exists an eigenfunction $\bw\in[\LO]^{2}$
of $\mathbf{T}$ associated with $\mu$ and $C>0$ such
that
\begin{equation}
\label{equat} 
\left\|\bw-\bw_h\right\|_{0,\O}\leq C h^{r_1}\left\{\|\bu-\bu_{h}\|_{1,\O}+\left[\sum_{E\in\T_{h}}\left(\left\| \bu-\bpiO\bu_{h}\right\|_{0,E}^{2}+\left| \bu-\bpisi\bu_{h}\right|_{1,E}^{2}\right)\right]^{1/2}\right\}.
\end{equation}
\end{lemma}

\begin{proof}
Thanks to Lemma~\ref{convnormnew}, Theorem~7.1 from \cite{BO} yields
spectral convergence of $\widehat{\mathbf{T}}_{h}$ to $\widetilde{\mathbf{T}}$.   In particular,  because of the
relation between the eigenfunctions of $\mathbf{T}$ and $\bT_h$ with those of $\widetilde{\mathbf{T}}$
and $\widehat{\mathbf{T}}_{h}$, respectively, we have that $\bw_h\in\CE_h$ and
there exists  $\bw\in\CE$ such that
\begin{equation*}
\left\|\bw-\bw_h\right\|_{0,\O}
\le C\sup_{\widetilde{{\bff}}\in\tE:\ \left\|\widetilde{{\bff}}\right\|_{0,\O}=1}
\big\|(\widetilde{\mathbf{T}}-\widehat{\mathbf{T}}_{h})\widetilde{{\bff}}\big\|_{0,\O}.
\end{equation*}
On the other hand, because of Lemma~\ref{lasttheo}, for all $\widetilde{{\bff}}\in\tE$,
if $\bff\in\CE$ is such that $\widetilde{{\bff}}=\bff$, then
$$
\big\|(\widetilde{\mathbf{T}}-\widehat{\mathbf{T}}_{h})\widetilde{{\bff}}\big\|_{0,\O}
=\left\|(\bT-\widehat{\mathbf{T}}_{h})\bff\right\|_{0,\O}
 \leq C h^{r_1}\left\{\|\bu-\bu_{h}\|_{1,\O}
 +\left[\sum_{E\in\T_{h}}\left(\left\| \bu-\bpiO\bu_{h}\right\|_{0,E}^{2}
 +\left| \bu-\bpisi\bu_{h}\right|_{1,E}^{2}\right)\right]^{1/2}\right\},
$$
which conclude the proof.
\end{proof}

We are now in a position to prove Theorem~\ref{asintotico}.

\begin{proof}
{\it of Theorem~\ref{asintotico}.}
Now, we are able to derive estimate~\eqref{meta}. With this aim,
we will bound each term on the right hand side of estimare \eqref{equat}
in Lemma~\ref{erroga}.

On the one hand, let $\bu\in\bV$ be the unique solution of the following problem. 
$$a(\bu,\bv)=b(\bw,\bv)\qquad \forall\bv\in\bV.$$ 
Since $a(\bw,\bv)=\l b(\bw,\bv)$ we have that $\bu=\bw/\l.$

On the other hand, let $\bu_{h}\in \bV_{h}$ be the unique solution of
the discrete problem:
\begin{equation}
\label{ro2}
a_{h}(\bu_{h},\bv_{h})=b_{h}(\bP_h\bw,\bv_{h})\qquad \forall \bv_{h}\in \bV_{h}.
\end{equation}
Now, since as stated above $\bu=\bw/\l$, we have that
\begin{equation}
\label{ro4}
\|\bu-\bu_{h}\|_{1,\O}\leq \dfrac{\|\bw-\bw_{h}\|_{1,\O}}{|\l|}+\left|\dfrac{1}{\l}-\dfrac{1}{\l_{h}}\right|\|\bw_{h}\|_{1,\O}+\left\|\dfrac{\bw_{h}}{\l_{h}}-\bu_{h}\right\|_{1,\O}.
\end{equation}
For the second term on the right hand side above, we use \eqref{cotavalor} to write 
\begin{align}
\label{ro5}
\left|\dfrac{1}{\l}-\dfrac{1}{\l_{h}}\right|=\dfrac{|\l-\l_{h}|}{|\l||\l_{h}|}&\leq C\left[\|\bw-\bw_{h}\|_{1,\O}^{2}+ \sum_{E\in\T_{h}}\left(\left\| \bw-\bpiO\bw_{h}\right\|_{0,E}^{2}+\left| \bw-\bpisi\bw_{h}\right|_{1,E}^{2}\right)\right].
\end{align}
In order to estimate the third term we recall first that
$$a_{h}(\bw_{h},\bv_{h})=\l_{h}b_{h}(\bw_{h},\bv_{h})\qquad \forall \bv_{h}\in \bV_{h}.$$
Then, subtracting this equation divided by $\l_{h}$ from \eqref{ro2} we have that 
$$a_{h}\!\left(\bu_{h}-\dfrac{\bw_{h}}{\l_{h}},\bv_{h}\right)=b_{h}(\bP_h\bw-\bw_{h},\bv_{h})\qquad \forall \bv_{h}\in \bV_{h}.$$
Hence, from the uniform ellipticity of $a_{h}(\cdot,\cdot)$ in $\bV_{h}$, we obtain
\begin{align*}
\left\|\bu_{h}-\dfrac{\bw_{h}}{\l_{h}}\right\|_{1,\O}^{2}&\leq C \|\bP_h\bw-\bw_{h}\|_{0,\O}\left\|\bu_{h}-\dfrac{\bw_{h}}{\l_{h}}\right\|_{0,\O}\leq C \|\bP_h(\bw-\bw_{h})\|_{0,\O}\left\|\bu_{h}-\dfrac{\bw_{h}}{\l_{h}}\right\|_{1,\O}.
\end{align*}
Therefore
\begin{align}
\label{ro6}
\left\|\bu_{h}-\dfrac{\bw_{h}}{\l_{h}}\right\|_{1,\O}&\leq C \|\bP_h\|\|\bw-\bw_{h}\|_{0,\O}\leq C \|\bw-\bw_{h}\|_{1,\O}.
\end{align}
Then, substituting \eqref{ro5} and  \eqref{ro6} into \eqref{ro4} we obtain 
\begin{align}
\label{ro7}
\|\bu-\bu_{h}\|_{1,\O}\leq C\left\{\|\bw-\bw_{h}\|_{1,\O}+\left[\sum_{E\in\T_{h}}\left(\left\| \bw-\bpiO\bw_{h}\right\|_{0,E}^{2}+\left| \bw-\bpisi\bw_{h}\right|_{1,E}^{2}\right)\right]^{1/2}\right\}.
\end{align}

For the second term on the right hand side of \eqref{equat} we have
\begin{align}
\label{roexl2}
 \left(\sum_{E\in\T_{h}}\left\| \bu-\bpiO\bu_{h}\right\|_{0,E}^{2}\right)^{1/2}\leq \|\bu-\bu_{h}\|_{1,\O}+ \left(\sum_{E\in\T_{h}}\left\| \bu_{h}-\bpiO\bu_{h}\right\|_{0,E}^{2}\right)^{1/2},
\end{align}
whereas  
\begin{align*}
\|\bu_{h}-\bpiO\bu_{h}\|_{0,E}\leq &C\left\|\bu_{h}-\dfrac{\bw_{h}}{\l_{h}}\right\|_{0,E}+\dfrac{\|\bw_{h}-\bpiO\bw_{h}\|_{0,E}}{\l_{h}}+\left\|\bpiO\left(\dfrac{\bw_{h}}{\l_{h}}-\bu_{h}\right)\right\|_{0,E}
\\&\leq  2\left\|\bu_{h}-\dfrac{\bw_{h}}{\l_{h}}\right\|_{0,E}+\dfrac{\|\bw-\bw_{h}\|_{0,E}}{\l_{h}}+\dfrac{\|\bw-\bpiO\bw_{h}\|_{0,E}}{\l_{h}}.\\
\end{align*}
Then, summing over all poligons and using \eqref{ro6}, we obtain
\begin{align*}
\left(\sum_{E\in\T_{h}}\left\| \bu_{h}-\bpiO\bu_{h}\right\|_{0,E}^{2}\right)^{1/2}&\leq  C\left[2\left\|\bu_{h}-\dfrac{\bw_{h}}{\l_{h}}\right\|_{1,\O}+\|\bw-\bw_{h}\|_{1,\O}+\left(\sum_{E\in\T_{h}}\left\| \bw_{h}-\bpiO\bw_{h}\right\|_{0,E}^{2}\right)^{1/2}\right]\\
&\leq C \left[\|\bw-\bw_{h}\|_{1,\O}+\left(\sum_{E\in\T_{h}}\left\| \bw_{h}-\bpiO\bw_{h}\right\|_{0,E}^{2}\right)^{1/2}\right].
\end{align*}

Substituting this and estimate \eqref{ro7} into \eqref{roexl2}  we obtain
\begin{equation}\label{newro}
\left(\sum_{E\in\T_{h}}\left\| \bu-\bpiO\bu_{h}\right\|_{0,E}^{2}\right)^{1/2}\leq C \left\{\|\bw-\bw_{h}\|_{1,\O}+\left[\sum_{E\in\T_{h}}\left(\left\| \bw_{h}-\bpiO\bw_{h}\right\|_{0,E}^{2}+\left| \bw_{h}-\bpisi\bw_{h}\right|_{1,E}^{2}\right)\right]^{1/2}\right\}.
\end{equation}

For the other term on the right hand side of \eqref{equat} we have
\begin{align}
\label{roex2}
 \left(\sum_{E\in\T_{h}}\left| \bu-\bpisi\bu_{h}\right|_{1,E}^{2}\right)^{1/2}\leq \|\bu-\bu_{h}\|_{1,\O}+ \left(\sum_{E\in\T_{h}}\left| \bu_{h}-\bpisi\bu_{h}\right|_{1,E}^{2}\right)^{1/2},
\end{align}
whereas  
\begin{align*}
|\bu_{h}-\bpisi\bu_{h}|_{1,E}&\leq  2\left\|\bu_{h}-\dfrac{\bw_{h}}{\l_{h}}\right\|_{1,E}+\dfrac{|\bw-\bw_{h}|_{1,E}}{\l_{h}}+\dfrac{|\bw-\bpisi\bw_{h}|_{1,E}}{\l_{h}}.
\end{align*}
Then, summing over all polygons and using \eqref{ro6}, we obtain
\begin{align*}
\left(\sum_{E\in\T_{h}}\left| \bu_{h}-\bpisi\bu_{h}\right|_{1,E}^{2}\right)^{1/2}&\leq  2\left\|\bu_{h}-\dfrac{\bw_{h}}{\l_{h}}\right\|_{1,\O}+\dfrac{\|\bw-\bw_{h}\|_{1,\O}}{\l_{h}}+\dfrac{\disp\left(\sum_{E\in\T_{h}}\left| \bw_{h}-\bpisi\bw_{h}\right|_{1,E}^{2}\right)^{1/2}}{\l_{h}}\\
&\leq C \left[\|\bw-\bw_{h}\|_{1,\O}+\left(\sum_{E\in\T_{h}}\left| \bw_{h}-\bpisi\bw_{h}\right|_{1,E}^{2}\right)^{1/2}\right].
\end{align*}

Substituting this and estimate \eqref{ro7} into \eqref{roex2}  we obtain
$$\left(\sum_{E\in\T_{h}}\left| \bu-\bpisi\bu_{h}\right|_{1,E}^{2}\right)^{1/2}\leq C \left\{\|\bw-\bw_{h}\|_{1,\O}+\left[\sum_{E\in\T_{h}}\left(\left\| \bw_{h}-\bpiO\bw_{h}\right\|_{0,E}^{2}+\left| \bw_{h}-\bpisi\bw_{h}\right|_{1,E}^{2}\right)\right]^{1/2}\right\}.$$

Finally, substituting the above estimate, \eqref{newro}
and \eqref{ro7} into \eqref{equat}, we conclude \eqref{meta} of Lemma \ref{asintotico}.
\end{proof}

\setcounter{equation}{0}
\section{A posteriori error estimator}
\label{SEC:aposteriori1}
The aim of this section is to introduce a suitable
residual-based error estimator for the elasticity equations
which is completely computable,
in the sense that it depends only on quantities available
from VEM solution. Then, we will  show its equivalence with the error.
For this purpose, we introduce the following definitions and notations.

For any polygon $E\in \CT_h$, we denote by $\CS_{E}$ the set of edges of $E$ and,
$$\CS=\bigcup_{E\in\CT_h}\CS_{E}.$$
 We decompose $\CS=\CS_{\O}\cup\CS_{\G_{D}}\cup\CS_{\G_N}$
where $\CS_{\G_{D}}=\{\ell\in \CS:\ell\subset \G_{D}\}$, $\CS_{\G_N}=\{\ell\in \CS:\ell\subset \G_N\}$
and $\CS_{\O}=\CS\backslash(\CS_{\G_{D}}\cup\CS_{\G_N})$.
For each edge $\ell\in \CS_{\O}$ and for any sufficiently
smooth function $\bv$, we define the following jump on $\ell$ by
$$\left[\!\left[ \Cten\beps(\bv) \bn\right]\!\right]_\ell:=\Cten\beps( \bv|_{E^{+}}) \boldsymbol{n}_{E^{+}}+\Cten\beps( \bv|_{E^{-}})\boldsymbol{n}_{E^{-}} .$$
 where $E^{+}$ and  $E^{-}$ are two element $\CT_{h}$ sharing the edge $\ell$ and $\boldsymbol{n}_{E^{+}}$ and $\boldsymbol{n}_{E^{-}}$ are the respective outer unit normal vectors.
 
As consequence of the mesh regularity assumptions,
we have that, each polygon $E\in\CT_h$, admits a sub-triangulation $\CT_h^{E}$
obtained by joining each vertex of $E$ with the midpoint of the ball with respect
to which $E$ is starred. Let $\hCT_h:=\bigcup_{E\in\CT_h}\CT_h^{E}$.
Since we are also assuming \textbf{A1} and \textbf{A2}, $\big\{\hCT_h\big\}_h$
is a shape-regular family of triangulations of $\O$. 

Now, we introduce bubble functions on polygons as follows.
A bubble function $\psi_{E}\in \H_0^1(E)$ for a polygon $E$
can be constructed piecewise as the sum of the cubic
bubble functions (cf. \cite{CGPS,Verfurth}) on each triangle of the
mesh element $\CT_h^{E}$. Now, an edge bubble function $\psi_{\ell}$ for
$\ell\in\partial E$ is a piecewise quadratic function,
attaining the value 1 at the barycenter of $\ell$ and vanishing
on the triangles $T\in\CT_h^{E}$  that do not contain $\ell$
on its boundary (see also \cite{CGPS}).

The following results which establish standard estimates
for bubble functions will be useful in what follows (see \cite{ATO,Verfurth}).

\begin{lemma}[Interior bubble functions]
\label{burbujainterior}
For any $E\in \CT_h$, let $\psi_{E}$ be the corresponding bubble function.
Then, there exists a constant $C>0$ 
independent of  $h_E$ such that
\begin{align*}
C^{-1}\|q\|_{0,E}^2&\leq \int_{E}\psi_{E} q^2\leq C\|q\|_{0,E}^2\qquad \forall q\in \bbP_k(E),\\
C^{-1}\| q\|_{0,E}&\leq \|\psi_{E} q\|_{0,E}+h_E\|\nabla(\psi_{E} q)\|_{0,E}\leq C\|q\|_{0,E}\qquad
\forall q\in \bbP_k(E).\\
\end{align*}
\end{lemma}
\begin{lemma}[Edge bubble functions]
\label{burbuja}
For any $E\in \CT_h$ and $\ell\in\partial E$, let $\psi_{\ell}$
be the corresponding edge bubble function. Then, there exists
a constant $C>0$ independent of $h_E$ such that
 \begin{align*}
C^{-1}\|q\|_{0,\ell}^2\leq \int_{\ell}\psi_{\ell} q^2& \leq C\|q\|_{0,\ell}^2\qquad
\forall q\in \bbP_k(\ell).\\
\end{align*}
Moreover, for all $q\in\bbP_k(\ell)$, there exists an extension of
$q\in\bbP_k(E)$ (again denoted by $q$) such that
 \begin{align*}
h_E^{-1/2}\|\psi_{\ell} q\|_{0,E}+h_E^{1/2}\|\nabla(\psi_{\ell} q)\|_{0,E}&\leq C\|q\|_{0,\ell}.\\
\end{align*}
\end{lemma}
\begin{remark}
\label{extencion}
A possible way of extending $q$ from $\ell\in\partial E$ to $E$
so that Lemma~\ref{burbuja} holds is  as follows:
first to extend $q$ to the straight line $L\supset\ell$ as the same polynomial function,
then to extend it to the whole plain through a constant
prolongation in the normal direction to $L$ and finally restricting it to $E$. 
\end{remark}

In what follows, let $(\l,\bw)$ be a solution to Problem~\ref{P1}.
We assume $\l$ is a simple eigenvalue and we normalize $\bw$
so that $\|\bw\|_{0,\O}=1$. Then, for each mesh $\CT_{h}$,
there exists a solution $(\l_{h},\bw_{h})$ of Problem~\ref{P3}
such that $\l_{h}\rightarrow\l$, $\|\bw_{h}\|_{0,\O}=1$
and $\|\bw-\bw_{h}\|_{1,\O}\rightarrow 0$ as $h\rightarrow 0$.

The following lemmas provide some error equations which will be the starting
points of our error analysis. First, we will denote with
$\be:=(\bw-\bw_{h})\in \bV$  the eigenfunction error and
we define the edge residuals as follows:
\begin{equation}
\label{saltocal}
J_{\ell}:=\left\{\begin{array}{l}
\dfrac{1}{2}\left[\!\left[ \Cten\beps(\bpisi  \bw_h) \bn\right]\!\right]_{\ell},\qquad\;\;\ell\in \CS_{\O},
\\[0.3cm]
-\Cten\beps(\bpisi  \bw_h) \bn, \qquad\qquad \;\ell\in\CS_{\G_N},
\\[0.3cm]
\0 , \qquad\qquad\qquad\qquad \;\ell\in\CS_{\G_D}.
\end{array}\right. 
\end{equation}
Notice that $J_{\ell}$ are actually computable since
they only involve values of $\bpisi  \bw_h\in [\bbP_{k}(E)]^{2}$ which is computable.
\begin{lemma}
\label{ext2}
For any $\bv\in \bV$, we have the following identity:
\begin{align*}
a(\be,\bv)&=\l b(\bw,\bv)-\l_h b(\bw_h,\bv)+\sum_{E\in \CT_h}\l_{h}b^E(\bw_h-\bpiO  \bw_h,\bv)-\sum_{E\in \CT_h}a^E(\bw_h-\bpisi  \bw_h,\bv)\\
&\quad+\sum_{E\in \CT_h}\left[\int_{E}\left(\l_{h}\varrho\bpiO \bw_{h}+\Div(\Cten\beps(\bpisi  \bw_h))\right)\cdot\bv+\sum_{\ell\in \CS_{E}}\int_{l}J_{\ell} \bv\right],
\end{align*}
where $\bpisi$ is the projector defined by \eqref{proje_0}.
\end{lemma}

\begin{proof} Using that $(\l,\bw)$ is a solution of Problem~\ref{P1},
adding and subtracting $\bpisi \bw_h$ and integrating by parts, we obtain the identity 
\begin{align*}
a(\be,\bv)&=\l b(\bw,\bv)-a(\bw_h,\bv)=\l b(\bw,\bv)-\sum_{E\in \CT_h}\left[a^E(\bw_h-\bpisi  \bw_h,\bv)+a^E(\bpisi  \bw_h,\bv)\right]\\
&=\l b(\bw,\bv)-\sum_{E\in \CT_h}a^E(\bw_h-\bpisi  \bw_h,\bv)-\sum_{E\in \CT_h}\left[-\int_{E}\Div(\Cten\beps(\bpisi  \bw_h))\cdot \bv+\int_{\partial E}\left(\Cten\beps(\bpisi  \bw_h) \bn \right)\cdot\bv\right]\\
&=\l b(\bw,\bv)-\l_{h} b(\bw_{h},\bv)+\sum_{E\in \CT_h}\l_{h}b^{E}(\bw_{h}-\bpiO  \bw_h\,\bv)-\sum_{E\in \CT_h}a^E(\bw_h-\bpisi  \bw_h,\bv)\\&
\quad+\sum_{E\in \CT_h}\left[\int_{E}(\l_{h}\varrho\bpiO  \bw_h+\Div(\Cten\beps(\bpisi  \bw_h)))\cdot \bv\right.\\\nonumber
&\quad-\left.\sum_{\ell \in \CS_{E}\cap(\CS_{\G_N})}\int_{\ell}\left(\Cten\beps(\bpisi  \bw_h) \bn \right)\cdot\bv+\dfrac{1}{2}\sum_{\ell \in \CS_{E}\cap\CS_{\O}}\int_{\ell}\left[\!\left[ \Cten\beps(\bpisi  \bw_h) \bn\right]\!\right]_{\ell}\bv\right].
\end{align*}
The proof is complete. 
\end{proof}

For all $E\in \CT_{h}$, we introduce the following local terms
and the local error indicator $\eta_{E}$ by:
\begin{align}
\label{RK}
\theta_{E}^2&:=b_h^E(\bw_h-\bpiO \bw_h,\bw_h-\bpiO  \bw_h)+a_h^E(\bw_h-\bpisi  \bw_h,\bw_h-\bpisi  \bw_h);\\
\label{etaK1}
R_{E}^{2}&:=h_{E}^{2}\| \l_{h} \varrho\bpiO\bw_{h}+\Div(\Cten\beps(\bpisi  \bw_h))\|_{0,E}^{2};\\
\label{etaK}
\eta_{E}^{2}&:= \theta_{E}^2+R_{E}^{2}+\sum_{\ell \in \CS_{E}}h_E\|J_{\ell}\|_{0,\ell}^2.   
\end{align}
Now, we are in a position to define the global error estimator by
\begin{align}
\label{etacald}
\eta&:=\left(\sum_{E\in \CT_h}\eta_{E}^2\right)^{1/2}.     
\end{align}

\begin{remark}
\label{RKK}
Contrary to the estimator obtained for standard finite
element approximations, in  the local estimator $\eta_{E}$,
for the virtual element approximations, appear the additional
term $\theta_{E}$. This term which represent the virtual
inconsistency of the VEM, has been also introduced
in \cite{BMm2as,CGPS} for a posteriori error estimates
of other VEM. Moreover,
we stress that the term $\theta_{E}$ can be directly
computed in terms of bilinear forms $S_{0}^{E}(\cdot,\cdot)$ and  $S_{\beps}^{E}(\cdot,\cdot)$.
In fact, 
\begin{align*}
\theta_{E}^2&=b_h^E(\bw_h-\bpiO \bw_h,\bw_h-\bpiO  \bw_h)+a_h^E(\bw_h-\bpisi  \bw_h,\bw_h-\bpisi  \bw_h)\\
&=S_{0}^{E}(\bw_h-\bpiO \bw_h,\bw_h-\bpiO  \bw_h)+S_{\beps}^E(\bw_h-\bpisi  \bw_h,\bw_h-\bpisi  \bw_h).
\end{align*}
\end{remark}

\subsection{Reliability of the a posteriori error estimator}

We now provide an upper bound for our error estimator.
\begin{theorem}
\label{erroipo}
There exists a constant $C>0$ independent of $h$ such that
\begin{align*}
\|\bw-\bw_h\|_{1,\O} &\leq C\left[\eta+\varrho\dfrac{(\l+\l_h)}{2}\|\bw-\bw_h\|_{0,\O}\right].
\end{align*}
\end{theorem}
\begin{proof}
For $\be=\bw-\bw_h\in \bV\subset [\HuO]^{2}$, there exists $\be_I\in \bV_{h}$
such that (see Lemma~\ref{estima4}),
\begin{equation}\label{cotainter}
\|\be-\be_{I}\|_{0,E}+h_{E}|\be-\be_{I}|_{1,E}\leq Ch_{E}\|\be\|_{1,E}.
\end{equation} 
Now, from Lemma~\ref{ext2}, we have that
\begin{align}\label{cotasalto}\nonumber
C \|\bw-\bw_h\|_{1,\O} ^{2}&\leq a(\bw-\bw_h,\be)=a(\bw-\bw_h,\be-\be_I)+a(\bw,\be_I)-a_h(\bw_h,\be_I)+a_h(\bw_h,\be_I)-a( \bw_h,\be_I)\\\nonumber
&=\underbrace{\l b( \bw,\be)-\l_h b(\bw_h,\be)}_{T_{1}}+\underbrace{\l_{h}\left[ b( \bw_{h},\be_{I})- b_{h}(\bw_h,\be_{I})\right]}_{T_2}+\underbrace{a_h(\bw_h,\be_I)-a( \bw_h,\be_I)}_{T_3}\\\nonumber
&\quad+\underbrace{\sum_{E\in \CT_h}\left[\l_{h}b^{E}(\bw_h-\bpiO  \bw_h, \be-\be_I)-a^{E}( \bw_h-\bpisi  \bw_h, \be-\be_I)\right]}_{T_{4}}\\
&\quad+\underbrace{\sum_{E\in \CT_h}\left[\int_E \left(\l_{h}\varrho\bpiO  \bw_h+\Div(\Cten\beps(\bpisi  \bw_h))\right)(\be-\be_I)+\sum_{\ell\in \CS_{E}}\int_{\ell}J_{\ell}(\be-\be_I)\right]}_{T_{5}}.
\end{align}
Now, we bound each term $T_{i}$, $i=1,\ldots,5$, with a constant $C$ independent of $h_{E}$.

First, we bound the term $T_1$, we use the definition of $b(\cdot,\cdot)$  and 
the fact that  $\|\bw\|_{0,\O}=\|\bw_h\|_{0,\O}=1$, we obtain
\begin{align}\label{cota3}
T_{1}&=\varrho\dfrac{(\l+\l_h)}{2}\|\be\|_{0,\O}^2\leq C\varrho\dfrac{(\l+\l_h)}{2}\|\be\|_{0,\O}\|\be\|_{1,\O}.
\end{align}

For the term $T_2$, we add and subtract $\bpiO  \bw_h$ on each $E\in\CT_h$,
and using the  {\it consistency} property  \eqref{consis1}, we have
\begin{align*}
T_2& \leq\l_{h}\left[\sum_{E\in \CT_h}b^E(\bw_h-\bpiO  \bw_h,\bw_h-\bpiO  \bw_h)^{1/2}b^E(\be_{I},\be_I)^{1/2}\right.\\
&\quad+\left.\sum_{E\in \CT_h}b_{h}^E(\bw_h-\bpiO  \bw_h,\bw_h-\bpiO  \bw_h)^{1/2}b_{h}^E(\be_{I},\be_I)^{1/2}\right]\\
& \leq C\sum_{E\in \CT_h}b_h^E(\bw_h-\bpiO  \bw_h,\bw_h-\bpiO  \bw_h)^{1/2}\|\be_{I}\|_{0,E}\\
& \leq C \left[\sum_{E\in \CT_h}b_h^E(\bw_h-\bpiO  \bw_h,\bw_h-\bpiO  \bw_h)\right]^{1/2}\|\be\|_{1,\O},
\end{align*}
where for the last estimate we have used the {\it stability} property  \eqref{stab2} and \eqref{cotainter}.

In a similar way, for the term $T_3$, we add and subtract
$\bpisi  \bw_h$ on each $E\in\CT_h$, using the  {\it consistency}
property  \eqref{consis0}, together a {\it stability}
property  \eqref{stab0} and \eqref{cotainter}, we have
\begin{align*}
T_3& \leq C \left(\sum_{E\in \CT_h}a_h^E(\bw_h-\bpisi  \bw_h,\bw_h-\bpisi  \bw_h)\right)^{1/2}\|\be\|_{1,\O}.
\end{align*}

To bound $T_4$, we use the  {\it stability} properties \eqref{stab0}
and \eqref{stab2} and \eqref{cotainter} to write
\begin{align*}
T_{4}&\leq \sum_{E\in \CT_h}\left[\l_{h}b^{E}(\bw_h-\bpiO  \bw_h, \be-\be_I)-a^{E}( \bw_h-\bpisi  \bw_h, \be-\be_I)\right]\\
&\leq C\left(\sum_{E\in \CT_h}\left[\l_{h}h_{E}b_{h}^{E}(\bw_h-\bpiO  \bw_h,\bw_h-\bpiO  \bw_h) +a_{h}^{E}( \bw_h-\bpisi  \bw_h,\bw_h-\bpisi  \bw_h)\right]\right)^{1/2}\|\be\|_{1,\O}.
\end{align*}
%
Therefore, by the above estimate and  \eqref{RK}, we have that 
 \begin{align}\label{cota6}
T_{2}+T_3+T_4 \leq C\left(\sum_{E\in \CT_h}\theta_{E}^{2}\right)^{1/2}\|\be\|_{1,\O}.
\end{align}

For the term $T_5$. First, we use a local trace inequality
(see \cite[Lemma~14]{BMRR}) and  \eqref{cotainter} to write
\begin{align*}
\|\be-\be_I\|_{0,\ell}&\leq C( h_E^{-1/2}\|\be-\be_I\|_{0,E}+h_E^{1/2}|\be-\be_I|_{1,E})\leq Ch_{E}^{1/2}\|\be\|_{1,E}.
\end{align*}
Hence, by the above inequality and \eqref{cotainter} again, we have,
\begin{align}\label{cot}
\nonumber
T_{5}&\leq C\sum_{E\in \CT_h}\left(\|\l_{h}\varrho\bpiO \bw_{h}+\Div(\Cten\beps(\bpisi  \bw_h))\|_{0,E}\|\be-\be_I\|_{0,E}+\sum_{\ell \in \CS_{E}}\|J_{\ell}\|_{0,\ell}\|\be-\be_I\|_{0,\ell}\right)\\\nonumber
&\leq C\sum_{E\in \CT_h}\left(h_E\|\l_{h}\varrho\bpiO \bw_{h}+\Div(\Cten\beps(\bpisi  \bw_h))\|_{0,E}\|\be\|_{1,E}+\sum_{\ell \in \CS_{E}}h_E^{1/2}\|J_{\ell}\|_{0,\ell}\|\be\|_{1,E}\right)\\
&\leq C\left[\sum_{E\in \CT_h}\left(h_E^2\|\l_{h}\varrho\bpiO \bw_{h}+\Div(\Cten\beps(\bpisi  \bw_h))\|_{0,E}^2+\sum_{\ell \in \CS_{E}}h_E\|J_{\ell}\|_{0,\ell}^2\right)\right]^{1/2}\|\be\|_{1,\O}.
\end{align}

Thus, the result follows from \eqref{cotasalto}--\eqref{cot}.
\end{proof}

The following result   establishes an estimate similar to the
above theorem for the projectors  $\bpiO$ and $\bpisi$.
\begin{corollary}
\label{corolario2}
There exists a constant $C>0$ independent of $h$ and $E$ such that:
\begin{equation*}
 \|\bw-\bw_h\|_{1,\O}+\left[\sum_{E\in \CT_h}\left(\|\bw-\bpiO \bw_h\|_{0,E}^{2}
 +|\bw-\bpisi\bw_{h}|_{1,E}^2\right)\right]^{1/2}\leq C\left[ \eta
 +\varrho\left(\dfrac{\l+\l_{h}}{2}\right)\|\bw-\bw_{h}\|_{0,\O}\right].
\end{equation*}
\end{corollary}
\begin{proof}
For each polygon $E\in\CT_h$, we have that
\begin{align*}
\|\bw-\bpiO \bw_{h}\|_{0,E}+|\bw-\bpisi\bw_{h}|_{1,E}&\leq
C\left(\|\bw- \bw_h\|_{1,E}+\|\bw_h-\bpiO \bw_h\|_{0,E}+|\bw_h-\bpisi  \bw_h|_{1,E}\right),
\end{align*}
then, summing over all polygons we obtain 
\begin{align*}
\sum_{E\in \CT_h}\left(\|\bw-\bpiO \bw_{h}\|_{0,E}^{2}+\|\bw-\bpisi  \bw\|_{1,E}^2\right)
&\leq C\left[\|\bw- \bw_h\|_{1,\O}^2 +\sum_{E\in \CT_h}\left(\|\bw_h- \bpiO\bw_h\|_{0,E}^2+|\bw_h-\bpisi  \bw_h|_{1,E}^2\right)\right].\\
\end{align*} 
Hence, from \eqref{stabilS} and \eqref{stabilS0}, together with Remark~\ref{RKK},
we have that $\|\bw_h- \bpiO\bw_h\|_{0,E}^2+|\bw_h-\bpisi  \bw_h|_{1,E}^2\leq C\theta_{E}^{2}\leq C\eta_{E}^{2}$.
Thus, the result follows from Theorem~\ref{erroipo}.
\end{proof}
We prove a convenient upper bound for the eigenvalue approximation. 
\begin{corollary}
\label{cotalambda}
There exists a constant $C>0$ independent of $h$ such that.
 \begin{align*}
|\l-\l_h|\leq C\left[\eta+ \varrho\left(\dfrac{\l+\l_{h}}{2}\right)\|\bw-\bw_{h}\|_{0,\O}\right]^{2}.
  \end{align*} 
  \end{corollary}
  \begin{proof}
  The result follows from Remark \ref{obserbapos} (see \eqref{cotavalor})
  and  Corollary~\ref{corolario2}.
  
%
\end{proof}

The upper bounds in Corollaries \ref{corolario2} and \ref{cotalambda}
are not computable since they involve the error term
$\|\bw-\bw_{h}\|_{0,\O}$. Our next goal is to prove that
this term is asymptotically negligible.

\begin{theorem}
There exist positive constants $C$ and $h_{0}$ such that, for all $h<h_{0}$, there holds
\begin{align}
\label{ro9}
& \|\bw-\bw_h\|_{1,\O}+\left[\sum_{E\in \CT_h}(\|\bw-\bpiO \bw_h\|_{0,E}^{2}+|\bw-\bpisi  \bw_{h}|_{1,E}^2)\right]^{1/2}\leq C\eta;\\
&|\l-\l_{h}|\leq C \eta^{2}.\label{ro10}
\end{align}
\end{theorem} 
\begin{proof}
From Theorem~\ref{asintotico} and Corollary~\ref{corolario2} we have
\begin{align*}
 \|\bw-\bw_h\|_{1,\O}+\left[\sum_{E\in \CT_h}(\|\bw-\bpiO \bw_h\|_{0,E}^{2}+|\bw-\bpisi  \bw_{h}|_{1,E}^2)\right]^{1/2}&\leq \\
&\hspace{-7cm}C\left( \eta+ h^{r}\left\{ \|\bw-\bw_h\|_{1,\O}+\left[\sum_{E\in \CT_h}\left(\|\bw-\bpiO \bw_h\|_{0,E}^{2}+|\bw-\bpisi  \bw_{h}|_{1,E}^2\right)\right]^{1/2}\right\}\right).
\end{align*}
Hence, it is straightforward to check that there exists $h_{0}>0$
such that for all $h<h_{0}$ \eqref{ro9} holds true.

On the other hand, from Lemma \ref{asintotico} and \eqref{ro9}
we have that for all $h<h_{0}$
$$\|\bw-\bw_{h}\|_{0,\O}\leq Ch^r\eta.$$
Then, for $h$ small enough, \eqref{ro10} follows
from Corollary~\ref{cotalambda} and the above estimate. 
\end{proof}

\subsection{Efficiency of the a posteriori error estimator}

In the present section we will show that the local error indicators
$\eta_{E}$ (cf. \eqref{etaK}) are efficient in the sense of pointing
out which polygons should be effectively refined.

First, we prove an upper estimate of the volumetric residual term $R_{E}$
introduced in \eqref{etaK1}.
\begin{lemma}
\label{eficiencia1}
There exists a constant $C>0$ independent of $h_E$, such that
\begin{equation*}
R_{E}\leq C \left(|\bw-\bw_{h}|_{1,E}+\theta_{E}+h_{E}\|\l \bw-\l_h \bw_h\|_{0,E}\right).
\end{equation*}
\end{lemma}
\begin{proof}
For any $E\in\CT_h$, let $\psi_{E}$ be the corresponding
interior bubble function,
we define  $\bv:=\psi_{E}\left(\l_{h}\bpiO \bw_{h}+\Div(\Cten\beps(\bpisi  \bw_h))\right)$. 
Since $\bv$ vanishes on
the boundary of  $E$. It may be extended by zero to the  whole domain $\O$. This extension, again denoted by $\bv$, belongs to $[\HuO]^{2}$ and from Lemma~\ref{ext2}, we have
\begin{align*}
a^{E}(\be,\bv)&=\l b^{E}(\bw,\bv)-\l_h b^{E}(\bw_h,\bv)+\l_{h}b^E(\bw_h-\bpiO  \bw_h,\bv)-a^E(\bw_h-\bpisi  \bw_h,\bv)\\
&\quad+\int_{E}\left(\l_{h}\bpiO \bw_{h}+\Div(\Cten\beps(\bpisi  \bw_h))\right)\cdot\bv.
\end{align*}
Since  $\left(\l_{h}\bpiO \bw_{h}+\Div(\Cten\beps(\bpisi  \bw_h))\right)\in [\bbP_k(E)]^{2}$,
using Lemma~\ref{burbujainterior} and the above equality, we obtain
\begin{align}\label{ghtjy}\nonumber
C^{-1}\|\l_{h}\bpiO \bw_{h}+\Div(\Cten\beps(\bpisi  \bw_h))\|_{0,E}^2&\leq \int_{E}\psi_{E}\left(\l_{h}\bpiO \bw_{h}+\Div(\Cten\beps(\bpisi  \bw_h))\right)^2\\\nonumber
&\hspace{-5.5cm}\leq C\left[\left(|\be|_{1,E}+|\bw_h-\bpisi  \bw_h|_{1,E}\right)\left|\psi_{E}(\l_{h}\bpiO \bw_{h}+\Div(\Cten\beps(\bpisi  \bw_h)))\right|_{1,E}\right.\\\nonumber
&\hspace{-5.0cm}+ \left. \left(\|\bw_h-\bpiO  \bw_h\|_{0,E}+\|\l \bw-\l_h \bw_h\|_{0,E}\right)\|\psi_{E}(\l_{h}\bpiO \bw_{h}+\Div(\Cten\beps(\bpisi  \bw_h)))\|_{0,E}\right]\\
&\hspace{-5.5cm}\leq Ch_{E}^{-1} \left[|\be|_{1,E}+\theta_{E}+h_{E}\left(\theta_{E}+\|\l \bw-\l_h \bw_h\|_{0,E}\right)\right]\|\l_{h}\bpiO \bw_{h}+\Div(\Cten\beps(\bpisi  \bw_h))\|_{0,E}.
\end{align}
where, for the  last estimate,  we have used again
Lemma~\ref{burbujainterior} together with  \eqref{stabilS}, \eqref{stabilS0} and Remark~\ref{RKK}.
Thus, multiplying the above inequality by $h_{E}$, allow us to conclude the proof.
\end{proof}
Next goal is to obtain an upper estimate
for the local term $\theta_{E}.$
\begin{lemma}
\label{cotaR}
There  exists $C>0$ independent of $h_E$ such that
\begin{equation*}
\theta_{E}\leq C\left(\|\bw-  \bw_{h}\|_{1,E}+\|\bw-\bpiO  \bw_h\|_{0,E}+|\bw-\bpisi  \bw_{h}|_{1,E}\right).
\end{equation*}
\end{lemma}
\begin{proof}
From definition of $\theta_{E}$,
together with Remark~\ref{RKK} and estimates  \eqref{stabilS} and \eqref{stabilS0}, we have 
\begin{align*}
\theta_{E}&\leq C\left(\|\bw_h-\bpiO  \bw_h\|_{0,E}+|\bw_h-\bpisi  \bw_h|_{1,E}\right)\\
&\leq C\left(\|\bw-  \bw_{h}\|_{1,E}+\|\bw-\bpiO  \bw_h\|_{0,E}+|\bw-\bpisi  \bw_{h}|_{1,E}\right).
\end{align*}
The proof is complete.
\end{proof}

The following lemma provides an upper estimate for
the jump terms of the local error indicator $\eta_{E}$ (cf. \eqref{etaK}).

\begin{lemma}
There exists a constant $C>0$ independent of $h_E$, such that
\label{lema4}
\begin{align}
\label{eqa}
h_E^{1/2}\left\|J_{\ell}\right\|_{0,\ell}&\leq C\left(| \bw-\bw_{h}|_{1,E}+\theta_{E}+h_{E}\|\l \bw-\l_h \bw_h\|_{0,E}\right)\quad\forall\ell\in\CS_{E}\cap\partial \O\neq\emptyset, \\\label{eqb}
h_E^{1/2}\left\|J_{\ell}\right\|_{0,\ell}&\leq C\left[\sum_{E'\in \omega_{\ell}}(|\be|_{1,E'}+\theta_{E'}+h_{E}\|\l \bw-\l_h \bw_h\|_{0,E'})\right]\qquad \forall\ell\in\CS_{E}\cap\CS_{\O},
\end{align}
where  $\omega_{\ell}:=\{E'\in\CT_{h}: \ell\subset\partial E'\}$.
\end{lemma}
\begin{proof}
First, for $\ell \in \CS_{E}\cap\CS_{\G_{D}}$,
we have $J_{\ell}=\0$, then \eqref{eqa} is obvious.

Secondly, for $\ell \in \CS_{E}\cap\CS_{\G_{N}}$,
we extend $J_{\ell}\in[\bbP_{k-1}(\ell)]^2$ to the element $E$ as in Remark~\ref{extencion}.
Let $\psi_{\ell}$ be the corresponding edge bubble function.
We define $\bv:=J_{\ell}\psi_{\ell}$. Then, $\bv$ may be extended by zero to the whole domain $\O$.
This extension, again denoted by $\bv$, belongs to $[\HuO]^2$ and from Lemma~\ref{ext2} we have  that
\begin{align*}
a^{E}(\be,\bv)&=\l b^{E}(\bw,J_{\ell}\psi_{\ell})-\l_h b^{E}(\bw_h,J_{\ell}\psi_{\ell})+b^E(\bw_h-\bpiO  \bw_h,J_{\ell}\psi_{\ell})-a^E(\bw_h-\bpisi  \bw_h,J_{\ell}\psi_{\ell})\\
&\quad+\int_{E}\left(\l_{h}\bpiO \bw_{h}+\Div(\Cten\beps(\bpisi  \bw_h))\right)\cdot J_{\ell}\psi_{\ell}+\int_{\ell}J_{\ell}^{2}\psi_{\ell}.
\end{align*}
For $J_{\ell}\in[\bbP_{k-1}(\ell)]^2$,
from Lemma~\ref{burbuja} and the above equality we obtain 
\begin{align*}
\left\|J_{\ell}\right\|^2_{0,\ell}&\leq \int_{\ell} J_{\ell}^2\psi_{\ell}\leq
C\left[\left(| \be|_{1,E}+| \bw_{h}-\bpisi\bw_{h}|_{1,E}\right)
\left|\psi_{\ell}J_{\ell}\right|_{1,E}\right.\\
&\quad+\left. \left(\|\l \bw-\l_h \bw_h\|_{0,E}+\|\l_{h}\bpiO \bw_{h}+\Div(\Cten\beps(\bpisi  \bw_h))\|_{0,E}+\|\bw_h-\bpiO  \bw_h\|_{0,E}
\right)\left\|J_{\ell}\psi_{\ell}\right\|_{0,E}\right]\\
&\leq C\left[\left(| \be|_{1,E}+\theta_{E}\right)
h_E^{-1/2}\left\|J_{\ell}\right\|_{0,\ell}+\left(\|\l \bw-\l_h \bw_h\|_{0,E}+(1+h_{E}^{-1})\theta_E
\right)h_E^{1/2}\left\|J_{\ell}\right\|_{0,\ell}\right]\\
&\leq Ch_E^{-1/2}\left\|J_{\ell}\right\|_{0,\ell} \left[|\be|_{1,E}+\theta_{E}+h_{E}\left(\theta_E+\|\l \bw-\l_h \bw_h\|_{0,E}\right)\right],
\end{align*}
where  we have used again Lemma \ref{burbuja} together
with estimate \eqref{ghtjy} of the proof of Lemma \ref{eficiencia1}.
Multiplying by $h_{E}^{1/2}$ the above inequality allows us to conclude \eqref{eqa}.

Finally, for $\ell \in \CS_{E}\cap\CS_\O$, we extend
$\bv:=J_{\ell}\psi_{\ell}$ to $[\HuO]^2$ as above  again. Taking into account that
$J_{\ell}\in[\bbP_{k-1}(\ell)]^2$ and $\psi_{\ell}$ is a quadratic bubble function in $E$,
from Lemma~\ref{ext2} we obtain
\begin{align*}
a(\be,\bv)&=\l b(\bw,J_{\ell}\psi_{\ell})-\l_h b(\bw_h,J_{\ell}\psi_{\ell})
+\sum_{E'\in \omega_{\ell}}b^{E'}(\bw_h-\bpiO  \bw_h,J_{\ell}\psi_{\ell})-\sum_{E'\in \omega_{\ell}}a^{E'}(\bw_h-\bpisi  \bw_h,J_{\ell}\psi_{\ell})\\
&+\sum_{E'\in \omega_{\ell}}\left(\int_{E'}\left(\l_{h}\bpiO \bw_{h}
+\Div(\Cten\beps(\bpisi  \bw_h))\right)\cdot J_{\ell}\psi_{\ell}+\int_{l}J_{\ell}^2\psi_{\ell}\right).
\end{align*}
Then, proceeding analogously to the above case we obtain
\begin{align*}
\|J_{\ell}\|^2_{0,\ell}&\leq Ch_{E}^{-1/2}\|J_{\ell}\|_{0,\ell}
\left[\sum_{E'\in \omega_{\ell}}(|\be|_{1,E'}+\theta_{E'}+h_{E}\|\l \bw-\l_h \bw_h\|_{0,E'})\right].
\end{align*}
Thus, the proof is complete.
\end{proof}

Now, we are in a position to prove the efficiency
of our local error indicator $\eta_E$. 
\begin{theorem}
\label{eficiencia}
There  exists $C>0$ such that
$$
\eta_{E}^2\leq 
C\left[\displaystyle\sum_{E'\in \omega_{E}}\left(\|\bw-\bw_h\|_{1,E'}^{2}
+\|\bw-\bpiO  \bw_h\|_{0,E'}+|\bw-\bpisi  \bw_{h}|_{1,E'}^{2}+h_{E}^{2}\|\l \bw-\l_h \bw_h\|_{0,E'}^{2}\right)\right],$$
$\text{where }\omega_{E}:=\{E'\in \CT_{h}: E' \text{ and } E\text{ share an edge }  \}$.
\end{theorem}
\begin{proof}
It follows immediately from Lemmas~\ref{eficiencia1}--\ref{lema4}.
\end{proof}

The following result establishes that term $h_{E}\|\l \bw-\l_h \bw_h\|_{0,E'}$
which appears in the above estimate is asymptotically negligible
for the global estimator $\eta$ (cf. \eqref{etacald}).

\begin{corollary}
There exists a constant $C>0$ such that
$$
\eta^2\leq C\left[\|\bw-\bw_h\|_{1,\O}^{2}+\sum_{E\in \CT_h}\left(\|\bw-\bpiO\bw_{h}\|_{0,E}^{2}+|\bw-\bpisi  \bw_{h}|_{1,E}^{2}\right)\right].
$$
\end{corollary}
\begin{proof}
From Theorem~\ref{eficiencia} we have that 
$$
\eta^2\leq C\left[\|\bw-\bw_h\|_{1,\O}^{2}+\sum_{E\in \CT_h}\left(\|\bw-\bpiO\bw_{h}\|_{0,E}^{2}+|\bw-\bpisi  \bw_{h}|_{1,E}^{2}\right)+h^{2}\|\l \bw-\l_h \bw_h\|_{0,\O}^{2}\right].$$
The last term on the right
hand side above is  bounded as follows: 
\begin{align*}
\|\l \bw -\l_h \bw_h\|_{0,\O}^2
&\leq 2 \l^{2}\|\bw-\bw_h\|_{0,\O}^{2}+2|\l-\l_h|^{2}\leq C\|\bw-\bw_h\|_{1,\O}^{2}+2|\l-\l_h|^{2},
\end{align*}
where we have used that $\|\bw_{h}\|_{0,\O}=1$.
Now, using the estimate \eqref{cotavalor},  we have
\begin{align*}
|\l-\l_{h}|^{2}&\leq (|\l|+|\l_{h}|)|\l-\l_{h}|\leq
C\left[\|\bw-\bw_h\|_{1,\O}^{2}+\sum_{E\in \CT_h}\left(\|\bw-\bpiO  \bw_h\|_{0,E}^{2}
+|\bw-\bpisi  \bw_h|_{1,E}^{2}\right)\right].
\end{align*} 
Therefore,
$$
\eta^2\leq C\left[\|\bw-\bw_h\|_{1,\O}^{2}+\sum_{E\in \CT_h}\left(\|\bw-\bpiO \bw_h\|_{0,E}^{2}
+\|\bw-\bpisi  \bw_h\|_{1,E}^{2}\right)\right]
$$
and we conclude the proof.
\end{proof}
\setcounter{equation}{0}
\section{Numerical results}
\label{SEC:NUMER}
We report in this section some numerical examples which have allowed us to assess
the theoretical result proved above. With this aim, we have
implemented in a MATLAB code a lowest-order VEM ($k=1$) on arbitrary
polygonal meshes following the ideas proposed in \cite{BBMR2014}.

To complete the choice of the VEM, we have to choose  the bilinear forms
$S_{\beps}^E(\cdot,\cdot)$ and $S_{0}^E(\cdot,\cdot)$ satisfying \eqref{stabilS}
and \eqref{stabilS0}, respectively. In this respect, we
have proceeded as in \cite[Section 4.6]{BBCMMR2013}: for each polygon $E$ with
vertices $P_1,\dots,P_{N_{E}}$, we have used
\begin{align*}
S_{\beps}^{E}(\bu,\bv)&:=\sigma_{E}\sum_{r=1}^{N_{E}}\bu(P_r)\bv(P_r),
\qquad \bu,\bv\in \bV^{E}_{h1}.\\
S_{0}^{E}(\bu,\bv)&:=\sigma_{E}^{0}\sum_{r=1}^{N_{E}}\bu(P_r)\bv(P_r),
\qquad \bu,\bv\in \bV^{E}_{h1},
\end{align*}
where $\sigma_{E}>0$ and $\sigma_{E}^{0}>0$
are multiplicative factors to take into account the magnitude of the
material parameter, for example, in the numerical tests a possible
choice could be to set $\sigma_{E}>0$ as the mean value of the eigenvalues
of the local matrix  $a^E(\bpisi\bu_h,\bpisi\bv_h)$ and for $\sigma_{E}^{0}>0$
as the mean value of the eigenvalues of the local matrix  $b^E(\bpiO\bu_h,\bpiO\bv_h)$.
This ensure that the stabilizing terms scales as $a^E(\bu_h,\bv_h)$
and $b^E(\bu_h,\bv_h)$, respectively. Finally, we mention that the above definitions
of the bilinear forms $S_{\beps}^{E}(\cdot,\cdot)$ and $S_{0}^{E}(\cdot,\cdot)$
are according with the analysis presented in \cite{MRR2015} in order to avoid spectral pollution.


\subsection{Test 1}

In this numerical test, we have taken an elastic body occupying the two dimensional
domain  $\O:=(0,1)^{2}$, fixed at its bottom $\G_{D}$
and free at the rest of boundary $\G_{N}$.
We have used different families of meshes
and the refinement parameter $N$ used to label
each mesh is the number of elements on each edge (see Figure~\ref{FIG:VM1}):
\begin{itemize}
\item $\CT_h^1$: trapezoidal meshes which consist of partitions
of the domain into $N\times N$ congruent trapezoids taking the
middle point of each edge as a new degree of freedom;
note that each element has 8 edges; 
\item $\CT_h^2$: non-structured hexagonal meshes made of convex hexagons.
\end{itemize}

We recall that the Lam\'e coefficients of a material are defined in
terms of the Young modulus $E_S$ and the Poisson ratio $\nu_S$ as follows:
$\lambda_S:=E_S\nu_S/[(1+\nu_S)(1-2\nu_S)]$ and $\mu_S:=E_S/[2(1+\nu_S)]$.
We have used the following physical parameters: density:
$\varrho=7.7\times 10^{3}$ kg/m$^{3}$, Young modulus:
$E_{S}=1.44\times 10^{11}$ Pa and Poisson ratio: $\nu_{S}=0.35$.

We observe that the eigenfunctions of this problem may present singularities at the points
where the boundary condition changes from Dirichlet ($\G_{D}$) to Neumann ($\G_{N}$).
According to \cite{G2}, for $\nu_{S}=0.35$, the estimate in Lemma~\ref{LEM:REG}(i)
holds true in this case for all $r < 0.6797$.
Therefore, the theoretical order of convergence for the vibration frequencies
presented in Theorem~\ref{convibration} is $2r\geq 1.36$
(see \cite{MMR2013} for further details).

\begin{figure}[H]
\begin{center}
\begin{minipage}{6.3cm}
\centering\includegraphics[height=4.1cm, width=4.1cm]{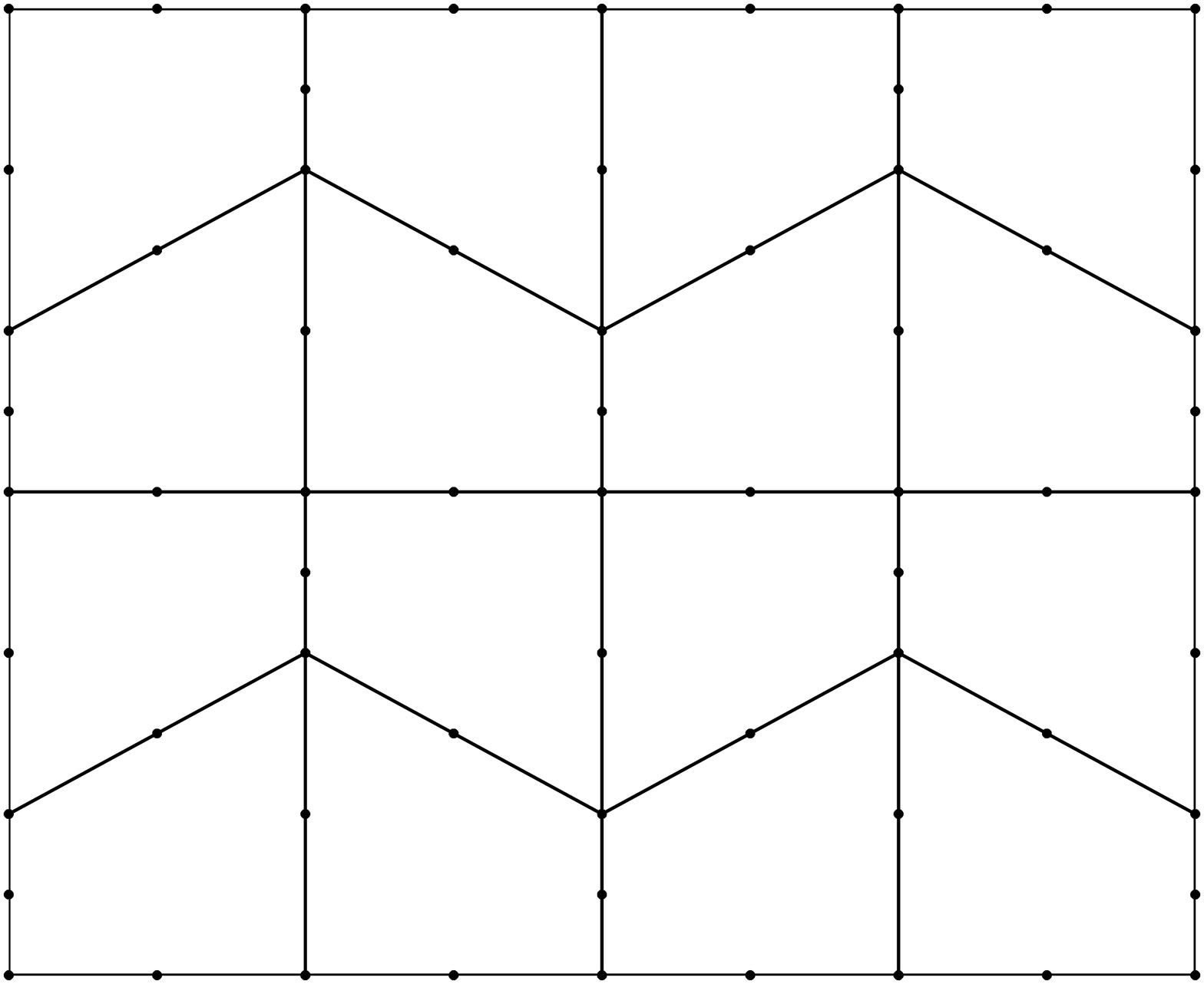}
\end{minipage}
\begin{minipage}{6.3cm}
\centering\includegraphics[height=4.1cm, width=4.1cm]{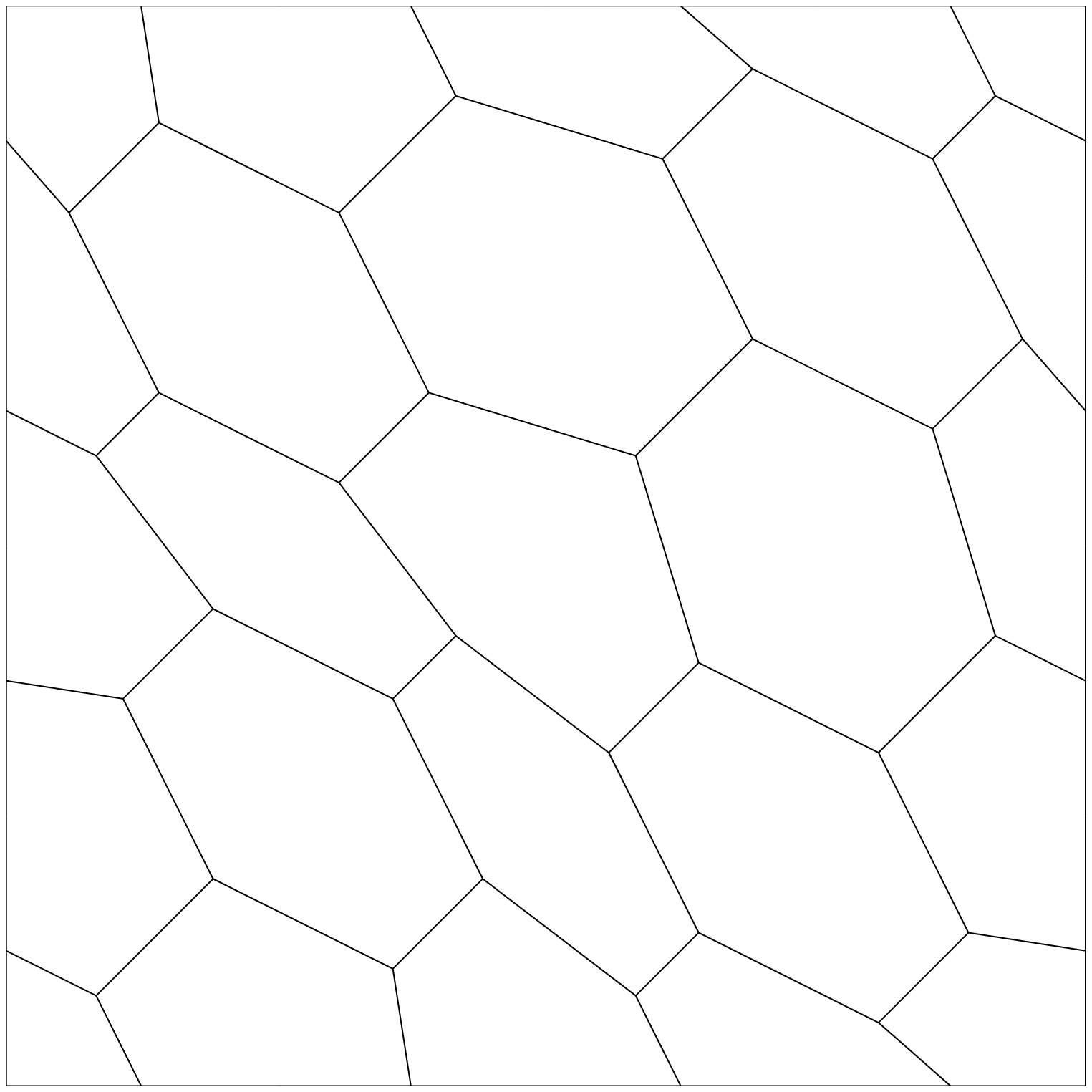}
\end{minipage}
\caption{ Sample meshes: $\CT_h^1$ and $\CT_h^2$ with $N=4$, respectively.}
 \label{FIG:VM1}
\end{center}
\end{figure}

We report in Table~\ref{TABLA:1}
the lowest vibration frequencies $\omega_{hi}:=\sqrt{\l_{hi}}$, $i=1,\ldots,6$
computed with the method analyzed in this paper. The table also includes
estimated orders of convergence, as  well as more accurate values of the
vibration frequencies extrapolated from the computed by means of
a least-squares fitting. Moreover, we compared our results with those
obtained in \cite{MMR2013} with a stress-rotation mixed formulation
of the elasticity system and a mixed Galerkin method based on AFW element.
With this aim, we include in the last column of Table~\ref{TABLA:1}
the values obtained by extrapolating those reported in \cite[Table~1]{MMR2013}.

\begin{table}[H]
\begin{center}
\caption{Test 1. Components lowest vibration frequencies $w_{hi}$, $i=1,\ldots,6$  on different meshes.}
\begin{tabular}{|c|c|c|c|c|c|c|c|c|c|}
\hline
  & Mesh&  $N=16$   & $N=32$ & $N=64$ & $N=128$ & Order &Extrapolated & \cite{MMR2013} \\
\hline
$\omega_{h1}$  & &2977.026  &  2955.750  &2948.391   &  2945.748 &  1.52    & 2944.387  &2944.295\\
$\omega_{h2} $ &  &  7386.910 &   7362.542  &  7353.758 &  7350.500 &    1.46   & 7348.674  &7348.840\\
$\omega_{h3}$  &$\CT_h^1$   &7992.109   & 7910.264  &   7888.147 &    7881.905    &   1.88  & 7879.746 &7880.084 \\
$\omega_{h4}$ & &13100.223 &  12838.752 &  12770.544 &  12752.434  & 1.93 & 12746.013   & 12746.802 \\
$\omega_{h5}$  &  &13289.395  & 13122.017 &  13072.453 &  13057.320 &  1.75 & 13051.220& 13051.758  \\
$\omega_{h6} $   & &15209.829 &  14975.380 &  14912.534 &  14895.790  & 1.90 &  14889.584 &14890.114\\
        \hline
  $\omega_{h1}$  & &2975.103  &  2955.754   &2948.274   &  2945.671 &  1.41    & 2943.964  &2944.295\\
$\omega_{h2} $ &  &  7383.823 &  7361.103  &  7353.189 &  7350.322 &    1.51   & 7348.834  & 7348.840\\
$\omega_{h3}$  &$\CT_h^2$   &8030.199    & 7921.047  &   7890.623 &   7882.914    &   1.87  & 7879.671 &7880.084 \\
$\omega_{h4}$ & &13174.876 &  12866.230  &  12778.890 &  12755.157   & 1.83 & 12745.302   & 12746.802 \\
$\omega_{h5}$  &  &13379.938   & 13149.980 &  13078.614&  13059.361 &  1.72 & 13049.282& 13051.758  \\
$\omega_{h6} $   & &15311.428 &  14997.597 &  14919.473 &  14897.987  & 1.98 &  14891.639 &14890.114\\
        \hline

\end{tabular}
\label{TABLA:1}
\end{center}
\end{table}

It can be seen from Table~\ref{TABLA:1}
that the eigenvalue approximation order
of our method is quadratic and that the results
obtained by the two methods agree perfectly well.
Let us remark that the theoretical order of convergence
($2r\geq 1.36$) is only a lower bound, since the actual order of
convergence for each vibration frequency depends on
the regularity of the corresponding eigenfunctions.
Therefore, the attained orders of convergence are in some cases
larger than this lower bound.

\subsection{Test 2}

The aim of this test is to assess the performance of the adaptive
scheme when solving a problem with a singular solution.
Let $\O:=[-0.75,0.75]^{2}\backslash[-0.5, 0.5]^{2}$
which corresponds to a two-dimensional closed vessel with vacuum inside.
The boundary of the elastic body is the union of $\G_{D}$ and $\G_{N}$: the solid is
fixed along $\G_{D}$ and free of stress along $\G_{N}$;
let $\boldsymbol{n}$ the  unit outward normal vector along
$\G_{N}$ (see Figure~\ref{FIG:interac}).
\begin{figure}[H]
\begin{center}
\input{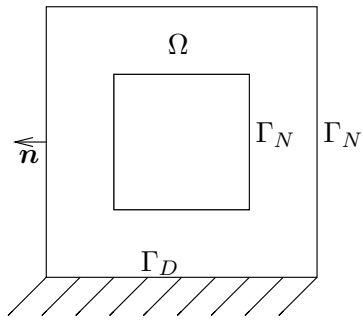}
\caption{Solid Domain.}
\label{FIG:interac}
\end{center}
\end{figure}
We have used the following physical parameters: density: $\varrho=1$ kg/m$^{3}$,
Young modulus: $E_{S}=1$ Pa and Poisson ratio: $\nu_{S}=0.35$.

In this numerical tests we have initiated the adaptive
process with a coarse triangular mesh. In order to compare the performance of
VEM with that of a the finite element method (FEM), we have used
two different algorithms to refine the meshes.
The first one is based on  a classical FEM strategy for which all
the subsequent  meshes consist of triangles. In such a case,
for $k=1$, VEM reduces to FEM. The other procedure to refine the
meshes is described in \cite{BMm2as}. It consists of splitting
each element into $n$ quadrilaterals ($n$ being the number of edges of the polygon) by connecting the
barycenter of the element with the midpoint  of each edge as shown
in Figure~\ref{FIG:cero} (see \cite{BMm2as} for more details).
Notice that although this process is initiated with a mesh of triangles,
the successively created meshes will contain other kind of convex  polygons as can be seen in Figure~\ref{FIG:VM35}.

\begin{figure}[H]
\centering
\subfigure[Triangle $E$ refined into 3 quadrilaterals.]{\includegraphics[height=3.6cm, width=3.6cm]{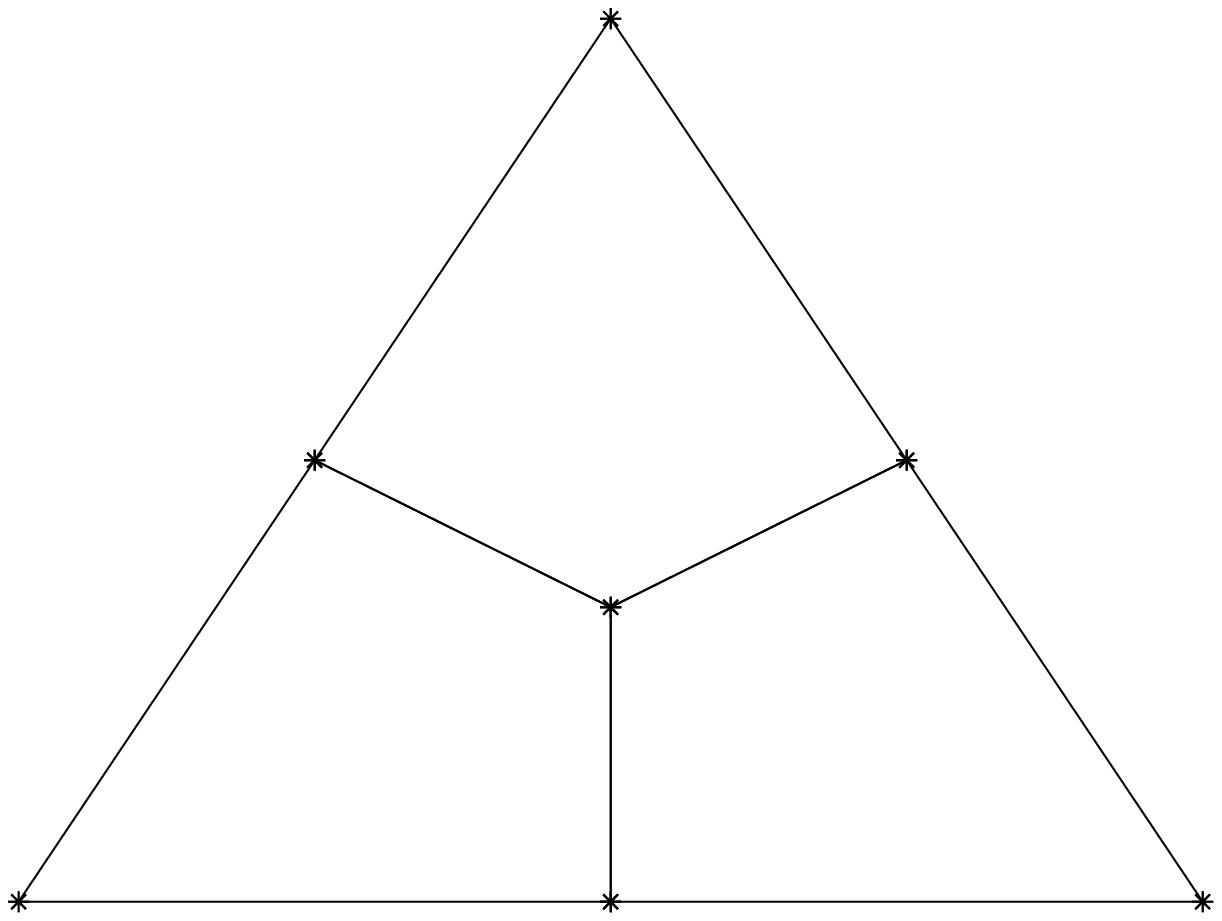}}
\hspace{0.5cm}\subfigure[Pentagon $E$ refined into 5 quadrilaterals.]{\includegraphics[height=3.6cm, width=3.6cm]{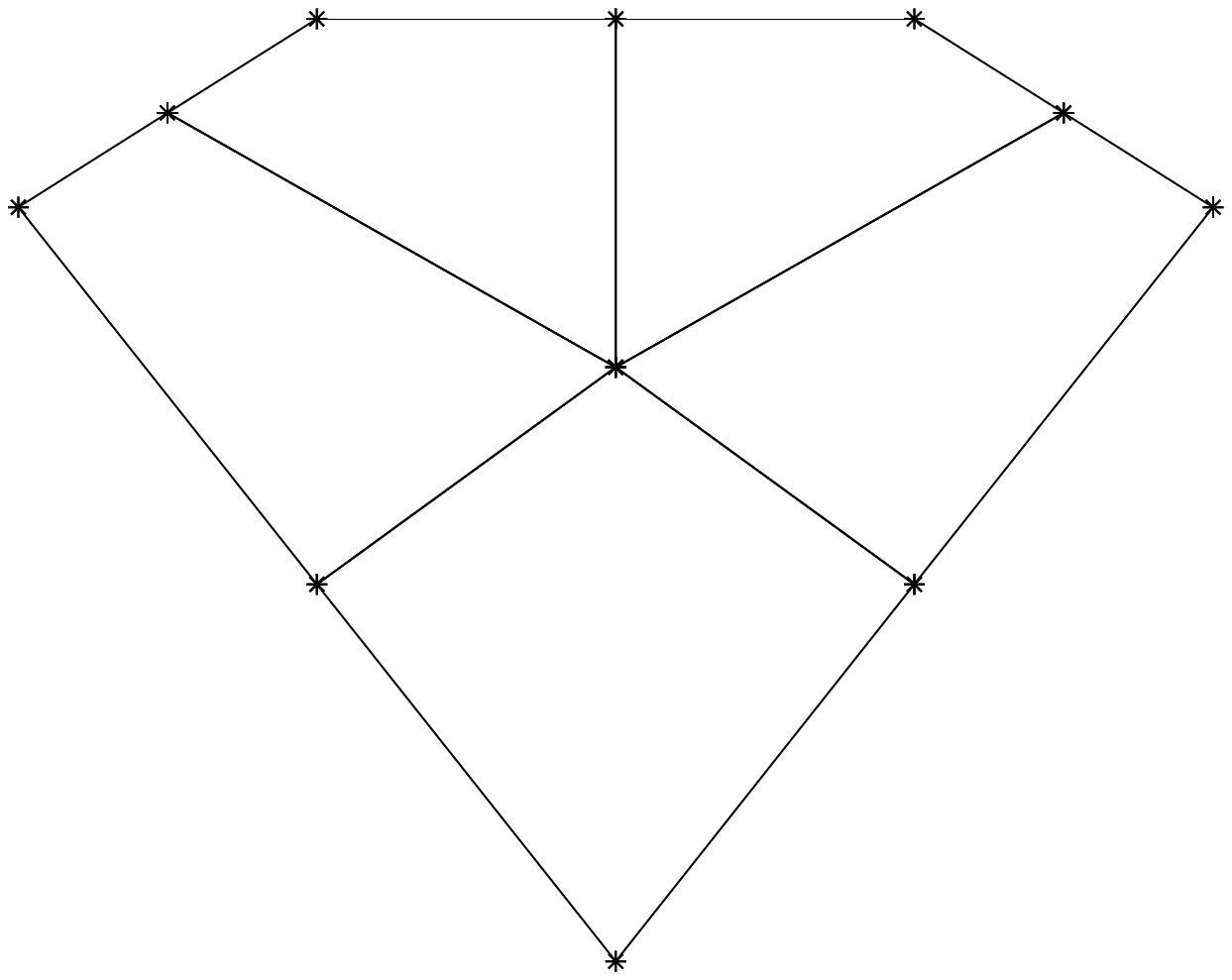}}
\caption{ Example of refined  elements for VEM strategy.}
 \label{FIG:cero}
\end{figure}

We have used the two refinement procedures (VEM and FEM)
described above. Both schemes are based on the strategy of refining those elements $E$ which satisfy
$$\eta_{E}\geq 0.5 \max_{E'\in \CT_{h}}\{\eta_{E'}\}.$$

Let us remark that in the case of triangular meshes, since
$\bV_{h1}^{E}=[\bbP_{1}(E)]^{2}$ and hence $\bpisi $ and $\bpiO$ are the identity, the term
$\theta_{E}^2$ (see \eqref{RK}) vanishes, by  the same reason, the projection $\bpisi $ also disappears
in the definition  \eqref{saltocal} of $J_{\ell}$ and $R_{E}$ in \eqref{etaK1}
reduces to $R_{E}^{2}=h_{E}^{2}\|\l_{h}\varrho\bw_{h}\|_{0,E}^{2}$.

The eigenfunctions of this problem may present singularities 
at the points where the boundary condition changes from Dirichlet
($\G_{D}$) to Neumann ($\G_{N}$) as well as at the reentrant angles of the domain
According to \cite{G2}, in this case, the estimate in Lemma~\ref{LEM:REG}(i)
holds true in this case for all $r < 0.5445$.
Therefore, in case of uniformly refined meshes,
the theoretical convergence rate for the eigenvalues should be
$|\l-\l_{h}|\simeq \mathcal{O}\left(h^{1.08}\right)\simeq\mathcal{O}\left(N^{-0.54}\right)$,
where $N$ denotes the number of degrees of freedom.
Now, an efficient adaptive scheme should lead to refine the meshes
in such a way that the optimal order $|\l-\l_{h}|\simeq\mathcal{O}\left(N^{-1}\right)$
could be recovered.

Figures \ref{FIG:VM351} and \ref{FIG:VM35} show the adaptively refined
meshes obtained with FEM and VEM procedures, respectively.
\begin{figure}[H]
\begin{center}
\begin{minipage}{4.0cm}
\centering\includegraphics[height=4.0cm, width=4.0cm]{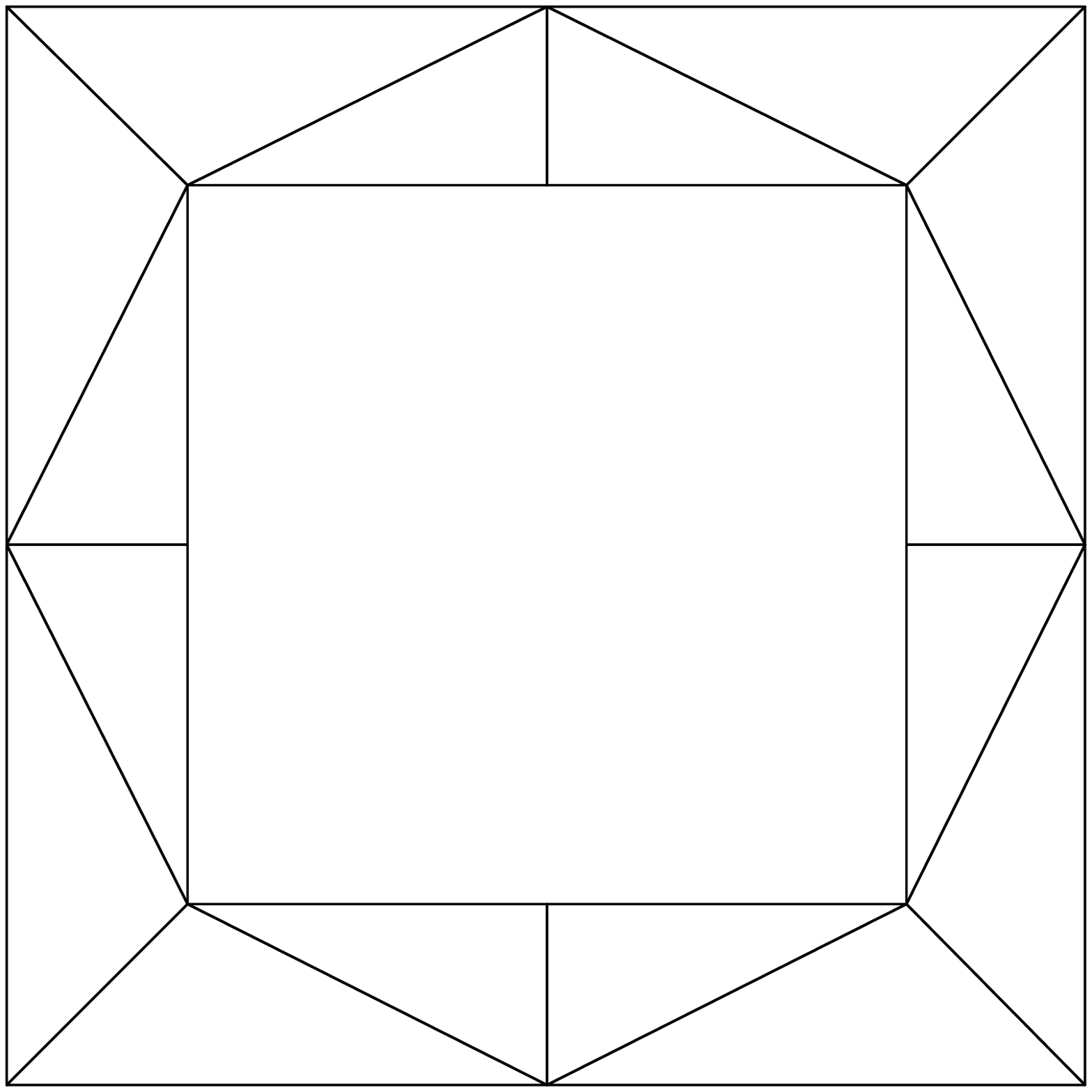}
\end{minipage}
\begin{minipage}{5.2cm}
\centering\includegraphics[height=5.2cm, width=5.2cm]{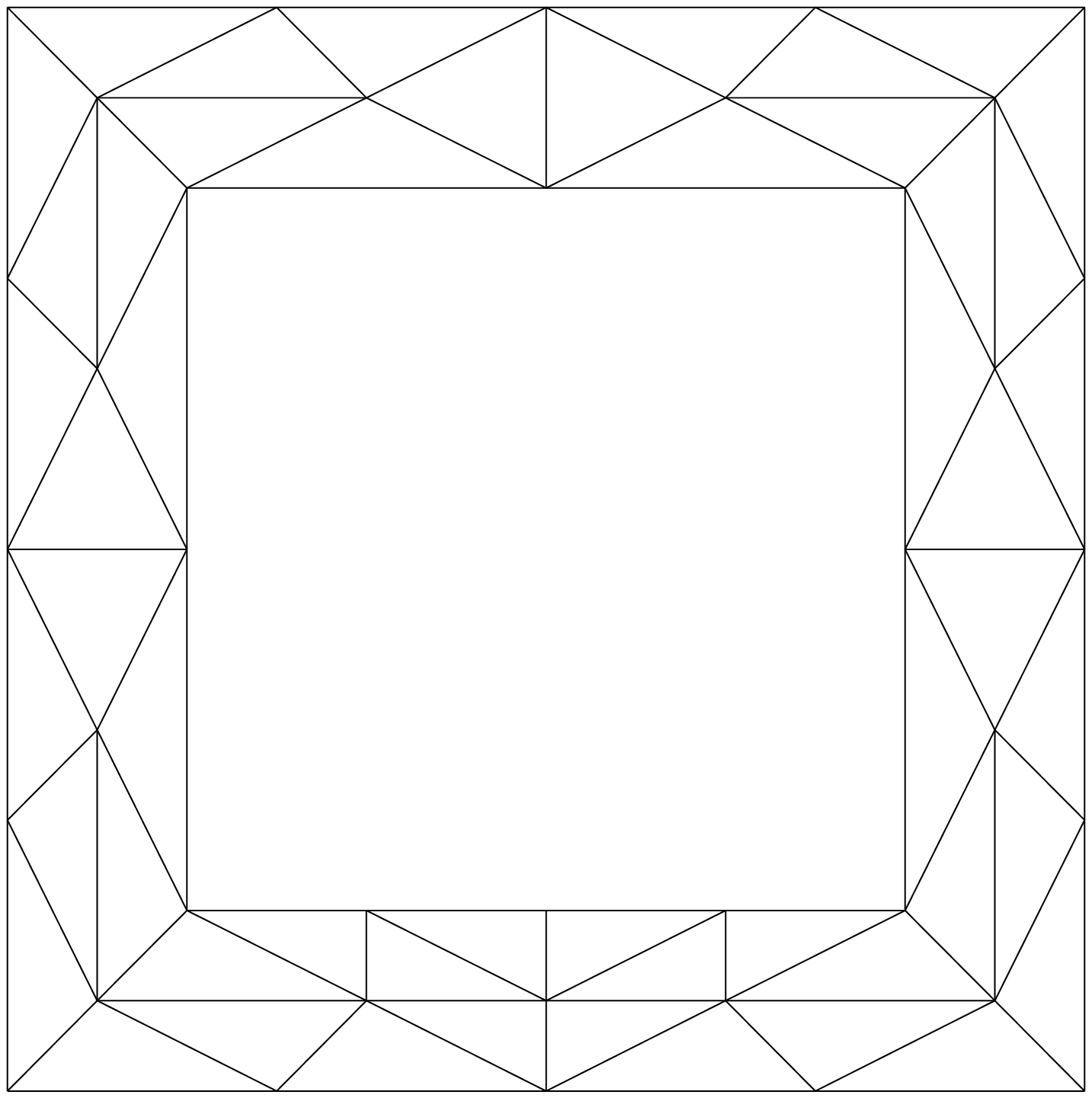}
\end{minipage}
\begin{minipage}{5.2cm}
\centering\includegraphics[height=5.2cm, width=5.2cm]{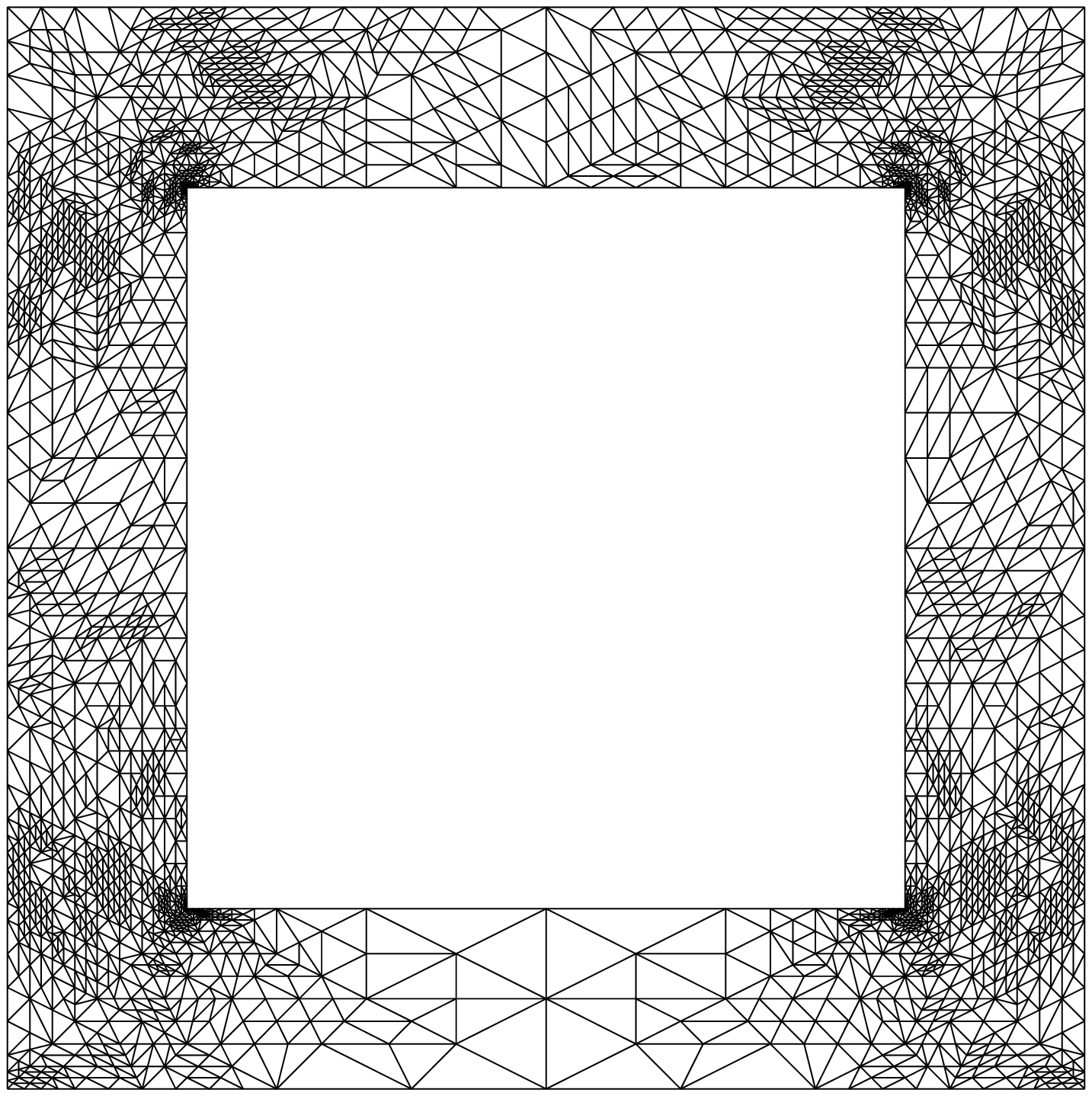}
\end{minipage}
\caption{Adaptively refined meshes obtained whit FEM scheme at refinement steps 0, 1 and 8.}
\label{FIG:VM351}
\end{center}
\end{figure}
\begin{figure}[H]
\begin{center}
\begin{minipage}{4.0cm}
\centering\includegraphics[height=4.1cm, width=4.1cm]{INICIAL.eps}
\end{minipage}
\begin{minipage}{5.2cm}
\centering\includegraphics[height=5.1cm, width=5.1cm]{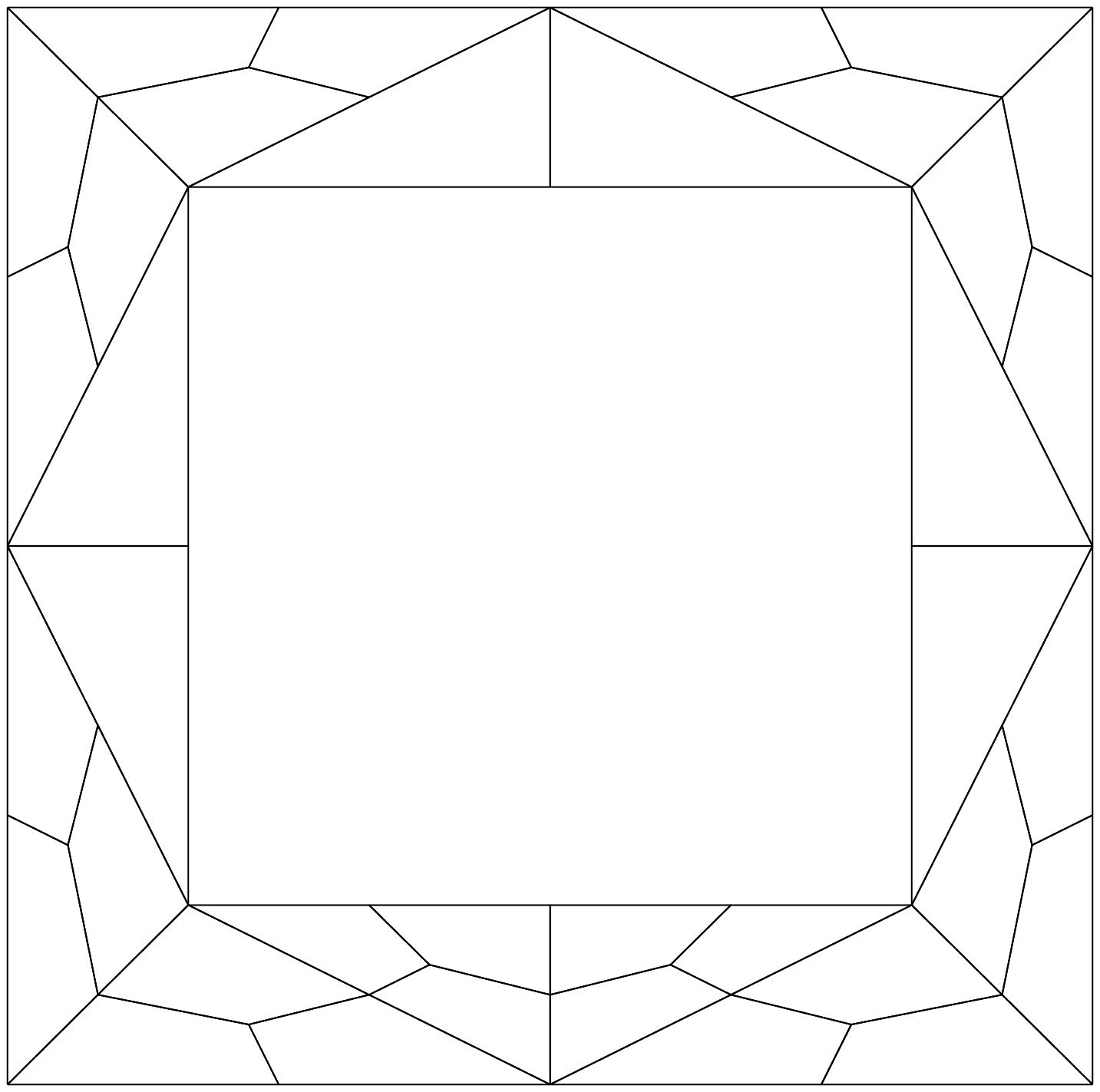}
\end{minipage}
\begin{minipage}{5.2cm}
\centering\includegraphics[height=5.1cm, width=5.1cm]{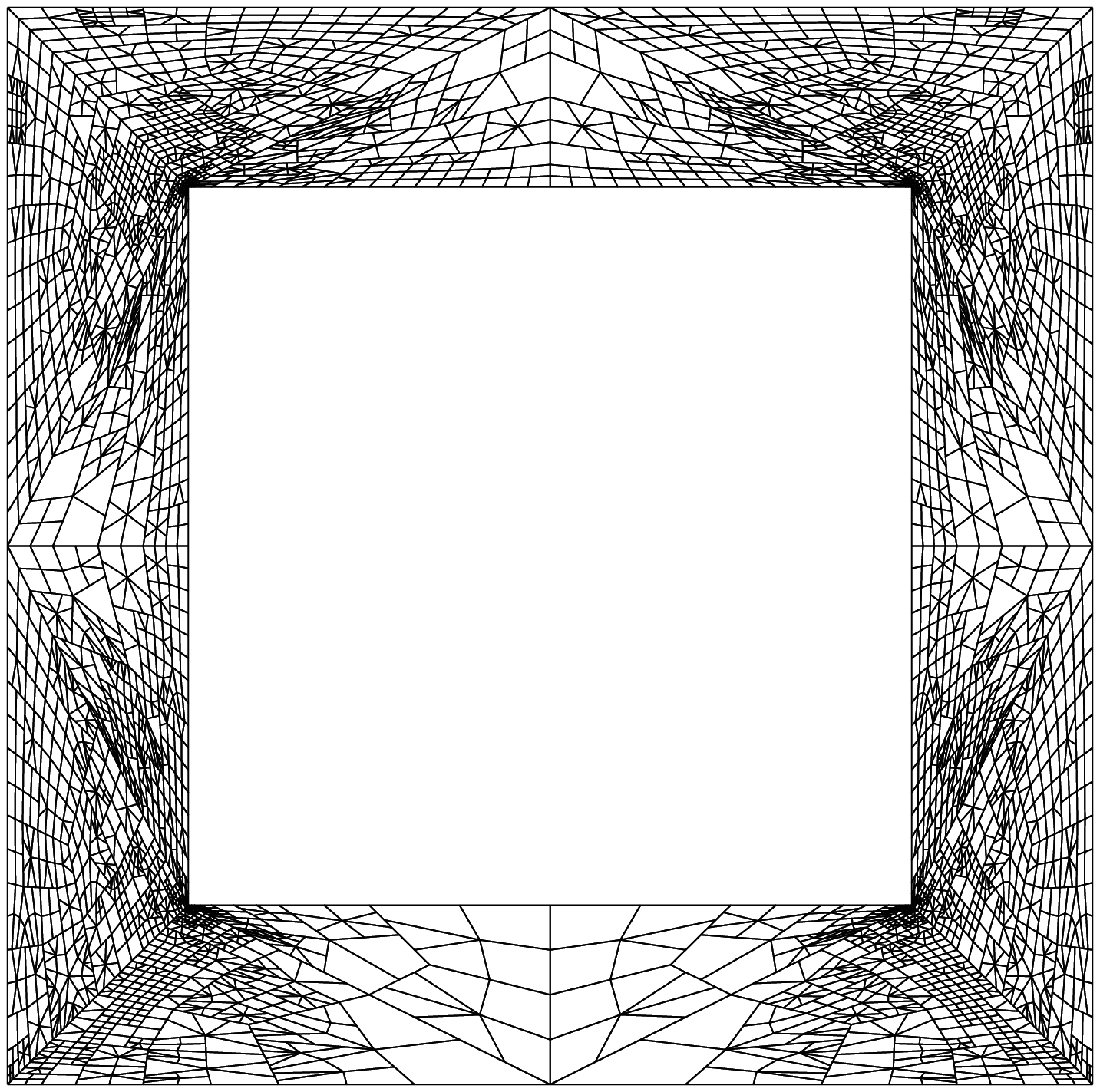}
\end{minipage}
\caption{Adaptively refined meshes obtained whit VEM scheme at refinement steps 0, 1 and 8.}
\label{FIG:VM35}
\end{center}
\end{figure}

In order to compute the errors $|\l_{1}-\l_{h1}|$,
due to the lack of an exact eigenvalue, we have used an approximation
based on a least squares fitting of the computed values obtained with
extremely refined meshes. Thus, we have obtained the value $\omega_{1}=\sqrt{\l_{1}}=0.1538$,
which has at least four correct significant digits.

We report in Table \ref{TABLA:5} the lowest vibration frequency
$\omega_{h1}$ on uniformly refined meshes and adaptive refined meshes
with FEM and VEM schemes. Each table includes the estimated convergence rate.

\begin{minipage}[t]{16.0cm}
\begin{table}[H]
\begin{center}
\caption{Test 2. frequency   $\omega_{h1}$ computed with different schemes:  uniformly refined meshes (``Uniform FEM''), adaptively refined meshes with  FEM (``Adaptive FEM'') and adaptively refined meshes with VEM (``Adaptive VEM'').}
\begin{tabular}{|c|c||c|c||c|c||c|c|}
  \hline
    \multicolumn{2}{|c||}{Uniform FEM}& \multicolumn{2}{|c||}{Adaptative FEM} &  \multicolumn{2}{|c||}{Adaptative VEM}  \\
    \hline
     $N$ & $\omega_{h1}$  &   $N$ & $\omega_{h1}$  &       $N$ & $\omega_{h1}$ \\
\hline
 136 &   0.2095   &  136  &  0.2095  &   136  &   0.2095\\ 
  390  &  0.1758  &  300 &   0.1810  &    340  &  0.1718\\  
  1418 &  0.1625 & 806  &  0.1659   &   646  &   0.1626\\
 5366  & 0.1567  &  1806  & 0.1599 &   1498 &  0.1574 \\
  20642 & 0.1551&  2946 &  0.1577 &   2942  & 0.1557 \\  
 80982  & 0.1543& 4198 &  0.1563  &  4788 &  0.1550  \\ 
            &             &  6348 &  0.1554 &   7782  & 0.1545 \\
           &              & 9000   &0.1549 &  12530& 0.1543 \\
           &              &  12894 &  0.1545 & 19398 &  0.1541  \\
          &               &18244 &0.1543 &  &      \\
         &                &    26760 & 0.1541 &  &     \\
    \hline 
     Order   &$\mathcal{O}\left(N^{-0.73}\right)$&    Order  &$\mathcal{O}\left(N^{-0.98}\right)$ &   Order   & $\mathcal{O}\left(N^{-1.0}\right)$\\
       \hline
       $\omega_1$  &0.1538 &      $\omega_1$  &0.1538 &   $\omega_1$  &0.1538\\
     \hline
    \end{tabular}
\label{TABLA:5}
\end{center}
\end{table}
\end{minipage}

\begin{figure}[H]
\begin{center}\begin{minipage}{8.0cm}
\centering\includegraphics[height=8.0cm, width=8.0cm]{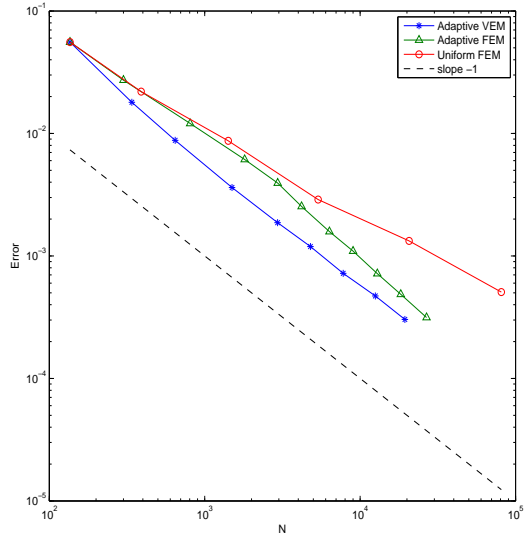}
\end{minipage}
\caption{Test2. Error curves of $|w_{1}-w_{h1}|$ for uniformly refined meshes
(``Uniform FEM''), adaptively refined meshes with
FEM (``Adaptive FEM'') and adaptively refined meshes with VEM (``Adaptive VEM'').}
\label{ero234}
\end{center}
\end{figure}

It can be seen from Figure~\ref{ero234} that the four refinement
schemes lead to the correct convergence rate. Moreover, the performance of adaptive VEM is slightly better than that of adaptive FEM.

We report in  Table \ref{TABLA:4}, the error $|w_{1}-w_{h1}|$
and the estimators $\eta^{2}$ at each step of the adaptative VEM scheme.
We include in the table the terms $\theta^{2}:=\sum_{E\in\CT_{h}}\theta_{E}^{2}$ which arise from the inconsistency of VEM, $R^{2}:=\sum_{E\in\CT_{h}}R_{E}^{2}$ which arise from the volumetric residuals and $J^{2}:=\sum_{E\in\CT_{h}}\left(\sum_{\ell\in\CT_{h}}h_{E}||J_{\ell}||_{0,\ell}^{2}\right)$ which arise from the edge residuals. We also report in the table the effectivity indexes $\dfrac{|\omega_{1}-\omega_{h1}|}{\eta^{2}}$.

\begin{table}[H]
\begin{center}
\caption{Test 2. Components of the error estimator and effectivity indexes on the adaptively refined meshes with VEM.}
\begin{tabular}{|c|c|c|c|c|c|c|c|c|}
\hline
$N$   & $\omega_{h1}$ &  $|\omega_{1}-\omega_{h1}|$   & $R^{2}$ & $\theta^{2}$ & $J^{2}$ &  $\eta^{2}$ & $\dfrac{|\omega_{1}-\omega_{h1}|}{\eta^{2}}$ \\
\hline
 136    & 2.095e-01 &  5.570e-02  & 2.795e-05 &  0       &     1.643e-01  & 1.643e-01  & 3.390e-01\\
   340   &  1.718e-01 &  1.797e-02 &  1.028e-05 &  2.244e-03&   3.501e-02 &  3.726e-02  & 4.823e-01\\
   646   &  1.626e-01 &  8.792e-03  & 4.353e-06   &1.874e-03  & 1.777e-02 &  1.965e-02   &4.475e-01\\
   1498 &   1.574e-01  & 3.623e-03 &  2.520e-06 &  9.645e-04&   7.441e-03 &  8.408e-03  & 4.309e-01\\
   2942  &  1.557e-01 &  1.872e-03  & 1.039e-06   &5.414e-04  & 4.348e-03 &  4.891e-03  & 3.827e-01\\
   4788  &  1.550e-01  & 1.194e-03 &  6.433e-07&   3.864e-04 &  2.883e-03 &  3.270e-03  & 3.652e-01\\
   7782  &  1.545e-01  & 7.216e-04  & 4.495e-07  & 2.472e-04  & 2.007e-03 &  2.255e-03  & 3.200e-01\\
   12530 &  1.543e-01 &  4.712e-04  & 2.894e-07   &1.682e-04 &  1.367e-03 &  1.536e-03  & 3.068e-01\\
   19398 &  1.541e-01 &  3.030e-04&   1.845e-07  & 1.155e-04  & 9.524e-04  & 1.068e-03 &  2.837e-01\\
 \hline
\end{tabular}
\label{TABLA:4}
\end{center}
\end{table}
It can be seen from the Table~\ref{TABLA:4} that the effectivity
indexes are bounded above and below far from zero and the inconsistency
and edge residual terms are roughly speaking of the same order,
none of them being asymptotically negliglible.

\section*{Acknowledgments}

The first author was partially supported by CONICYT-Chile through
FONDECYT project 1140791 and by DIUBB through project 171508 GI/VC,
Universidad del B\'io-B\'io, (Chile).
The second author was partially supported by a CONICYT-Chile
through FONDECYT initiation project 111170534.
%
The authors are deeply grateful Prof. Rodolfo Rodr\'iguez (Universidad de
Concepci\'on) for the fruitful discussions.


\bibliographystyle{amsplain}

\begin{thebibliography}{24}

\bibitem{equiv} 
\textsc{B. Ahmad, A. Alsaedi, F. Brezzi, L.D. Marini and A. Russo},
\textit{Equivalent projectors for virtual element methods},
Comput. Math. Appl., 66, (2013), pp.  376--391.


\bibitem{ATO} 
\textsc{M. Ainsworth  and J.T. Oden},
\textit{A posteriori error estimation in finite element analysis},
In: Pure and Applied Mathematics. Wiley, New York (2000).

\bibitem{ABMV2014} 
\textsc{P.F. Antonietti, L. Beir\~ao da Veiga, D. Mora and M. Verani},
\textit{A stream virtual element formulation of the Stokes problem on polygonal meshes},
SIAM J. Numer. Anal., 52(1), (2014), pp. 386--404.

\bibitem{ABSV2016} 
\textsc{P.F. Antonietti, L. Beir\~ao da Veiga, S. Scacchi and M. Verani},
\textit{A $C^1$ virtual element method for the Cahn--Hilliard equation with polygonal meshes},
SIAM J. Numer. Anal., 54(1), (2016), pp. 36--56.






\bibitem{BO} 
\textsc{I. Babu\v{s}ka and J. Osborn}, 
\textit{Eigenvalue problems}, 
in \textit{Handbook of Numerical Analysis}, Vol. II, 
P.G. Ciarlet and J.L. Lions, eds., 
North-Holland, Amsterdam, 1991, pp. 641--787.



\bibitem{BBCMMR2013} 
\textsc{L. Beir\~ao da Veiga, F. Brezzi, A. Cangiani, G. Manzini,
L.D. Marini and A. Russo},
\textit{Basic principles of virtual element methods},
Math. Models Methods Appl. Sci., 23, (2013), pp. 199--214.

\bibitem{BBM} 
\textsc{L.~Beir\~ao~da Veiga, F. Brezzi and L.D. Marini},
\textit{Virtual elements for linear elasticity problems},
SIAM J. Numer. Anal., 51, (2013), pp. 794--812.

\bibitem{BBMR2014} 
\textsc{L. Beir\~ao da Veiga, F. Brezzi, L.D. Marini and A. Russo},
\textit{The hitchhiker's guide to the virtual element method},
Math. Models Methods Appl. Sci., 24, (2014), pp. 1541--1573.



\bibitem{BBMRm3as2016} 
\textsc{L. Beir\~ao da Veiga, F. Brezzi, L.D. Marini and A. Russo},
\textit{Virtual Element Method for general second-order
elliptic problems on polygonal meshes},
Math. Models Methods Appl. Sci., 26(4), (2016), pp. 729--750.


\bibitem{BLMbook2014} 
\textsc{L. Beir\~ao da Veiga, K. Lipnikov and G. Manzini},
\textit{The Mimetic Finite Difference Method for Elliptic Problems},
Springer, MS\&A, vol. 11, 2014.

\bibitem{BLM2015} 
\textsc{L. Beir\~ao da Veiga, C. Lovadina and D. Mora},
\textit{A virtual element method for elastic and
inelastic problems on polytope meshes},
Comput. Methods Appl. Mech. Engrg., 295, (2015) pp. 327--346.

\bibitem{BLV-M2AN} 
\textsc{L. Beir\~ao da Veiga, C. Lovadina  and G. Vacca},
\textit{Divergence free virtual elements for
the Stokes problem on polygonal meshes},
ESAIM Math. Model. Numer. Anal., 51(2), (2017) pp. 509--535.



\bibitem{BMm2as} 
\textsc{L. Beir\~ao da Veiga and G. Manzini},
\textit{Residual a posteriori error estimation for the
virtual element method for elliptic problems},
ESAIM Math. Model. Numer. Anal., 49 (2015), pp. 577--599.

\bibitem{BMRR} 
\textsc{L. Beir\~ao da Veiga, D. Mora, G. Rivera and R. Rodr\'iguez},
\textit{A virtual element method for the acoustic vibration problem},
Numer. Math., 136(3), (2017) pp. 725--763.


\bibitem{BMR2016} 
\textsc{L. Beir\~ao da Veiga, D. Mora and G. Rivera},
\textit{Virtual elements for a shear-deflection
formulation of Reissner-Mindlin plates},
Math. Comp., DOI: https://doi.org/10.1090/mcom/3331 (2017).






\bibitem{BBBPS2016} 
\textsc{M.F. Benedetto, S. Berrone, A. Borio, S. Pieraccini, S. Scial\`o},
\textit{Order preserving SUPG stabilization for the virtual element
formulation of advection--diffusion problems},
Comput. Methods Appl. Mech. Engrg., 311, (2016), pp. 18--40.






\bibitem{BGHRS2008} 
\textsc{A. Berm\'udez, P. Gamallo, L. Hervella-Nieto, R. Rodr\'iguez and D. Santamarina},
\textit{Fluid-structure acoustic interaction.
Computational Acoustics of Noise Propagation in Fluids. Finite and Boundary Element Methods},
S. Marburg, B. Nolte, eds. Springer, 2008, Chap. 9, pp. 253--286.



\bibitem{BHPR2001} 
\textsc{A. Berm\'udez and R. Rodr\'iguez},
\textit{Finite element computation of the vibration modes of a fluid-solid system},
Comput. Methods Appl. Mech. Engrg., 119(3-4), (1994), pp. 355--370.


\bibitem{BeBo2017} 
\textsc{S. Berrone and A. Borio},
\textit{A residual a posteriori error estimate for the virtual element method},
Math. Models Methods Appl. Sci., 27, (2017), pp. 1423--1458.


\bibitem{Boffi} 
\textsc{D. Boffi},
\textit{Finite element approximation of eigenvalue problems},
Acta Numerica, 19, (2010), pp. 1--120.

\bibitem{BGG2012} 
\textsc{D. Boffi, F. Gardini and L. Gastaldi}, 
\textit{Some remarks on eigenvalue approximation by finite elements}, 
in \textit{Frontiers in numerical analysis--Durham 2010}, 
Lect. Notes Comput. Sci. Eng., 85, Springer, Heidelberg, (2012), pp. 1--77.

\bibitem{BS-2008} 
\textsc{S.C. Brenner and R.L. Scott},
\textit{The Mathematical Theory of Finite Element Methods},
Springer, New York, 2008.



\bibitem{BM12} 
\textsc{F. Brezzi and L.D. Marini},
\textit{Virtual elements for plate bending problems},
Comput. Methods Appl. Mech. Engrg., 253, (2012), pp. 455--462.

\bibitem{CG16} 
\textsc{E. C\'aceres and G.N. Gatica},
\textit{A mixed virtual element method for the pseudostress-velocity
formulation of the Stokes problem},
IMA J. Numer. Anal., 37(1), (2017) pp. 296--331.

\bibitem{CGS17} 
\textsc{E. C\'aceres, G.N. Gatica and F. Sequeira},
\textit{A mixed virtual element method
 for the Brinkman problem},
Math. Models Methods Appl. Sci., 27(4), (2017) pp. 707--743.


\bibitem{CGPS} 
\textsc{A. Cangiani, E.H. Georgoulis, T. Pryer and O.J. Sutton},
\textit{A posteriori error estimates for the virtual element method},
Numer. Math., 137(4), (2017) pp. 857--893.


\bibitem{CGH14} 
\textsc{A. Cangiani,  E.H.  Georgoulis and P. Houston},
\textit{$hp$-version discontinuous Galerkin methods on polygonal and polyhedral meshes},
Math. Models Methods Appl. Sci., 24(10), (2014), pp. 2009--2041.
 

\bibitem{ChM-camwa} 
\textsc{C. Chinosi and L.D. Marini},
\textit{Virtual element method for fourth order problems: $\mathrm{L}^2$-estimates,}
Comput. Math. Appl., 72(8), (2016), pp. 1959--1967.


%
%
%


\bibitem{DPECMAME2015} 
\textsc{D. Di Pietro and A. Ern},
\textit{A hybrid high-order locking-free method for linear elasticity on general meshes},
Comput. Methods Appl. Mech. Eng., 283, (2015), pp. 1--21.

\bibitem{DPECRAS2015} 
\textsc{D. Di Pietro and A. Ern},
\textit{Hybrid high-order methods for
variable-diffusion problems on general meshes},
C. R. Acad. Sci., Paris I, 353(1), (2015), pp. 31--34.


 
\bibitem{Paulino-VEM}
\textsc{A.L. Gain, C. Talischi and G.H. Paulino},
\textit{On the virtual element method for three-dimensional
linear elasticity problems on arbitrary polyhedral meshes},
Comput. Methods Appl. Mech. Engrg., 282, (2014), pp. 132--160.

\bibitem{GVXX}
\textsc{F. Gardini and G. Vacca},
\textit{Virtual element method for second order
elliptic eigenvalue problems},
IMA J. Numer. Anal., DOI: https://doi.org/10.1093/imanum/drx063 (2017).


\bibitem{GR} 
\textsc{V. Girault and P.A. Raviart},
\textit{Finite Element Methods for Navier-Stokes Equations},
Springer-Verlag, Berlin, 1986.


\bibitem{G2}
\textsc{P. Grisvard},
\textit{Probl\'ems aux limites dans les polygones. Mode d\'emploi}, 
EDF, Bull. Dir. Etudes Rech. Ser. C, 1, (1986), pp. 21--59.




\bibitem{Hernadez2009} 
\textsc{E. Hern\'andez},
\textit{Finite element approximation of the elasticity spectral problem on curved domains},
J. Comput. Appl. Math., 225, (2009), pp. 452--458.



\bibitem{K} 
\textsc{T. Kato}, 
\textit{Perturbation Theory for Linear Operators},
Springer Verlag, Berlin, 1995.



\bibitem{MMR2013} 
\textsc{S. Meddahi, D. Mora and R. Rodr\'iguez},
\textit{Finite element spectral analysis for the mixed formulation of the elasticity equations},
SIAM J. Numer. Anal., 51(2), (2010), pp. 1041--1063.

\bibitem{MRR2015} 
\textsc{D. Mora, G. Rivera and R. Rodr\'iguez},
\textit{A virtual element method for the Steklov eigenvalue problem},
Math. Models Methods Appl. Sci., 25(8), (2015), pp. 1421--1445.
 
 \bibitem{MRR2} 
\textsc{D. Mora, G. Rivera and R. Rodr\'iguez},
\textit{A posteriori error estimates for a virtual element
method for the Steklov eigenvalue problem},
Comput. Math. Appl., 74(9), (2017), pp. 2172--2190.
 
 

\bibitem{MRV} 
\textsc{D. Mora, G. Rivera and I. Vel\'asquez},
\textit{A virtual element method for the vibration problem of Kirchhoff plates},
ESAIM Math. Model. Numer. Anal., DOI: https://doi.org/10.1051/m2an/2017041 (2017).

\bibitem{PPR15} 
\textsc{I. Perugia, P. Pietra and A. Russo},
\textit{A plane wave virtual element method for the Helmholtz problem},
ESAIM Math. Model. Numer. Anal., 50(3), (2016), pp. 783--808.





\bibitem{ST04} 
\textsc{N. Sukumar and A. Tabarraei},
\textit{Conforming polygonal finite elements},
Internat. J. Numer. Methods Engrg., 61, (2004), pp. 2045--2066.

\bibitem{TPPM10}
\textsc{C. Talischi, G.H. Paulino, A. Pereira and I.F.M. Menezes},
\textit{Polygonal finite elements for topology optimization: A unifying paradigm},
Internat. J. Numer. Methods Engrg., 82(6), (2010), pp. 671--698.

\bibitem{vacca1} 
\textsc{G. Vacca},
\textit{Virtual Element Methods for hyperbolic problems on polygonal meshes},
Comput. Math. Appl., 74(5), (2017), pp. 882--898.

\bibitem{V-m3as18} 
\textsc{G. Vacca},
\textit{An $H^1$-conforming virtual element for Darcy and Brinkman equations},
Math. Models Methods Appl. Sci., 28(1), (2018), pp. 159--194.


\bibitem{vacca2}
\textsc{G. Vacca and L. Beir\~ao da Veiga},
\textit{Virtual element methods for parabolic problems on polygonal meshes},
Numer. Methods Partial Differential Equations, 31(6), (2015), pp. 2110--2134.

\bibitem{Verfurth} 
\textsc{R. Verfurth},
\textit{A review of a posteriori error estimate
and adaptative mesh-refinement techniques},
Wiley-Teubner, Chichester (1996).

\bibitem{WRR2016}
\textsc{P. Wriggers, W.T. Rust and B.D. Reddy},
\textit{A virtual element method for contact},
Comput. Mech., 58, (2016), pp. 1039--1050.














\end{thebibliography}

\end{document}